\documentclass[10pt]{elsarticle} % use larger type; default would be 10pt

\usepackage{ushort}
\usepackage[utf8]{inputenc} % set input encoding (not needed with XeLaTeX)
\usepackage{subfigure}
\usepackage{geometry} % to change the page dimensions
\usepackage{slashbox}
\usepackage{graphicx} % support the \includegraphics command and options
\usepackage{here}
\usepackage{amsmath} \allowdisplaybreaks
\usepackage{booktabs} % for much better looking tables
\usepackage{array} % for better arrays (eg matrices) in maths
\usepackage{paralist} % very flexible & customisable lists (eg. enumerate/itemize, etc.)
\usepackage{verbatim} % adds environment for commenting out blocks of text & for better verbatim
\usepackage{subfig} % make it possible to include more than one captioned figure/table in a single float
\usepackage{fancyhdr} % This should be set AFTER setting up the page geometry
\usepackage{sectsty}
\sectionfont{\normalsize\bfseries\boldmath}
\subsectionfont{\normalsize\normalfont\sffamily}

\usepackage{amsmath}
\usepackage{amssymb}
\usepackage{dsfont}
\usepackage{bbold}
\usepackage{multirow}
\newcommand{\dt}[1]{\frac{\partial#1}{\partial t}}
\newcommand{\ddt}[1]{\frac{\partial^2 #1}{\partial t^2}}
\newcommand{\Lnorm}[2]{\left(#1, #2\right)}
\newcommand{\Enorm}[2]{a\left(#1, #2\right)}

\newcommand{\llnorm}[1]{\left\lVert #1\right\rVert_{L_2(\Omega)}}

\newcommand{\ilnorm}[1]{\left\lVert #1\right\rVert_{L_\infty(0,T;L_2(\Omega))}}
\newcommand{\Llnorm}[1]{\left\lVert #1\right\rVert_{L_2(0,T;L_2(\Omega))}}
\newcommand{\hnorm}[1]{\left\lVert #1\right\rVert_{H^1(\Omega)}}

\newcommand{\enorm}[1]{\left\lVert #1\right\rVert_{V}}

\newcommand{\gnorm}[1]{\left\lVert #1\right\rVert_{L_2(\Gamma_N)}}
\newcommand{\ignorm}[1]{\left\lVert #1\right\rVert_{L_\infty(0,T;L_2(\Gamma_N))}}
\newcommand{\lgnorm}[1]{\left\lVert #1\right\rVert_{L_2(0,T;L_2(\Gamma_N))}}
\newcommand{\hgnorm}[1]{\left\lVert #1\right\rVert_{H^1(0,T;L_2(\Gamma_N))}}
\newcommand{\sobolevl}[1]{\vertiii{v}_{H^0(\mathcal{E}_h)}}

\usepackage{amsthm}
\newtheorem{lemma}{Lemma}[section]
\newtheorem{theorem}{Theorem}[section]
\newtheorem{corollary}{Corollary}[section]
\numberwithin{equation}{section}
\journal{Journal of Computational and Applied Mathematics}

\begin{document}
	
	\begin{frontmatter}
	\title{Finite Element Approximation and Analysis of a Viscoelastic Scalar Wave
		Equation with Internal Variable Formulations}
	\author[1]{Yongseok Jang\corref{cor1}%
		\fnref{fn1}}
		\ead{jang@cerfacs.fr}   % \ead{yongseok.jang@brunel.ac.uk}
	\author[2]{Simon Shaw}
	\ead{simon.shaw@brunel.ac.uk}
	\cortext[cor1]{Corresponding author}
	\fntext[fn1]{Jang
		gratefully acknowledges the support of a scholarship from
		Brunel University London.}
	\address[1]{CERFACS, 42 Avenue Gaspard Coriolis, 31057, Toulouse, France.}
	\address[2]{Department of Mathematics, Brunel University London, Uxbridge UB8 3PH, UK.}
\begin{abstract}
\noindent We consider linear scalar wave equations with a hereditary
integral term of the kind used to model viscoelastic solids. 
The kernel in this Volterra integral is a sum of decaying exponentials
(The so-called Maxwell, or Zener model) and this allows the introduction
of one of two types of families of internal variables, each of which evolve
according to an ordinary differential equation (ODE). There is one such ODE
for each decaying exponential, and the introduction of these ODEs means
that the Volterra integral can be removed from the governing equation.
The two types of internal variable are distinguished by whether the
unknown appears in the Volterra integral, or whether its time
derivative appears; we call the resulting problems the
\emph{displacement} and \emph{velocity} forms.
We define fully discrete formulations for each of these forms by
using continuous Galerkin finite element approximations in space and an implicit 
`Crank-Nicolson' type of finite difference method in time. We prove stability
and \textit{a priori} bounds, and using the FEniCS environment, \texttt{https://fenicsproject.org/}
(The FEniCS project version 1.5, Archive of Numerical Software, 3 (100), 9--23, 2015.)
give some numerical results. These bounds do not require Gr\"onwall's
inequality and so can be regarded to be of high quality, allowing confidence in
long time integration without an \textit{a priori} exponential build up
of error. As far as we are aware this is the first
time that these two formulations have been described together
with accompanying proofs of such high quality stability and error bounds.
The extension of the results to vector-valued viscoelasticity problems is 
straightforward and summarised at the end. The numerical results
are reproducible by acquiring the python sources from
\texttt{https://github.com/Yongseok7717}, or by running a custom built docker 
container (instructions are given). 
 
\end{abstract}


\end{frontmatter}

\section{Introduction}\label{sec:intro}\noindent
Materials that exhibit both elastic and viscous response to imposed load and/or
deformation are called viscoelastic. Typical examples of such solid materials are
amorphous polymers, soft biotissue, metals at high temperatures and even concrete
\cite{hunter1976mechanics}. The mathematical description of the dynamic response
of these materials uses a momentum balance law to relate external forces, $f$, to
acceleration, $\ddot{u}$, and stress divergence, $\nabla\cdot\ushort\sigma$. A faithful
mathematical model of this physical set-up would require the introduction of a
vector-valued partial differential equation as in elastodynamics (see e.g.
\cite{VE,DGV}) but we restrict ourselves here to a scalar analogue to keep
the exposition as simple as possible (but we use the terminology of solid 
mechanics).

So, with that in mind, once boundary and initial data are specified, the physical
problem is exemplified by the following mathematical model: find 
$u\colon[0,T]\times\Omega\to\mathbb{R}$ such that
\begin{eqnarray}
\rho \ddot{ u}-\nabla\cdot\ushort\sigma=f&\textrm{in}& (0,T]\times\Omega,
\label{eq:momeq}
\\
u=0&\textrm{on}&[0,T]\times\Gamma_D,
\\
\label{eq:BCdisp}
\ushort\sigma\cdot\ushort n=  g_N&\textrm{on}&[0,T]\times\Gamma_N,
\label{eq:BCstress}
\\
u=u_0&\textrm{on}&\{0\}\times\Omega,
\label{eq:ICdisp}
\\
\dot{ u}=w_0&\textrm{on}&\{0\}\times\Omega,\label{eq:initial:velo}
\end{eqnarray}
where we refer to $\ushort\sigma$ as the stress (and describe it more fully below);
$\Omega$ is an open bounded polytopic domain in $\mathbb{R}^d$ with
constant mass density $\rho$; $\Gamma_D$ and $\Gamma_N$ are the `Dirichlet' and `Neumann' boundaries and $T>0$ is a final time.
As usual $\Gamma_D$ and $\Gamma_N$ are disjoint and we will assume
that the surface measure of $\Gamma_D$ is strictly positive. Note that we use overdots to denote
time differentiation so that $\dot{u}:=\dt{u}$ and $\ddot{u}:=\ddt{u}$. In classical continuum mechanics, the physical model is defined with a displacement vector so the stress is a second order tensor and is defined by a constitutive relationship with the strain tensor. Hence, in general, the linear viscoelastic dynamic equation is a vector-valued PDE of which the above is a scalar analogue. However, \eqref{eq:momeq} is not only a scalar analogue but also represents the mathematical model of viscoelastic materials subjected to antiplane shear response. Antiplane strain, for instance, allows us to reduce the second order tensor to a vector so that the viscoelastic antiplane model in 3D can be dealt with by the scalar wave problem in 2D (see e.g. \cite{Barber2004,paulino2001viscoelastic,hoarau2002analysis}).

The viscoelasticity literature contains a large number of \emph{rheological} (spring and dashpot) based phenomenological models
(e.g. the Maxwell, Voigt, Kelvin-Voigt, generalised Maxwell, $\ldots$, models --- see
\cite{VE,FLO} for more details) as well as models based on the fractional calculus, referred to
as `power law' models in \cite{GGa}. The spring and dashpot models are the ones of interest
here because they give rise to stress-strain constitutive laws that can be described by Volterra
kernels of sums of decaying exponentials. This, in turn, makes them much better
suited to numerical approximation than the fractional calculus models in the sense that the
entire solution history need not be stored, and there is no weak singularity in the kernel.
We will return to the first point below, but first recall the form of these constitutive laws
from \cite{VE} in the following two equivalent (integrate by parts) forms:
\begin{align}
\ushort\sigma(t)
=&D\varphi(0)\nabla u(t)-\int^t_0D\varphi_s(t-s)\nabla u(s)ds,
\label{eq:dispconlaw}
\\
=&D\varphi(t)\nabla u(0)+\int^t_0D\varphi(t-s)\nabla \dot u(s)ds,
\label{eq:veloconlaw}
\end{align}
where $D$ is a positive constant, $\varphi(t)$ is a stress relaxation function and $\varphi_s(t-s):=\frac{\partial}{\partial s}\varphi(t-s)$. Now we can complete our model problem by defining the stress relaxation function $\varphi(t)$ and then substituting for $\ushort{\sigma}$
in \eqref{eq:momeq} and \eqref{eq:BCstress}. The generalised Maxwell model for a viscoelastic solid
produces the following stress relaxation function (see e.g. \cite{GGa}),
\begin{equation}
\varphi(t)=\varphi_0+\sum_{q=1}^{N_\varphi}\varphi_qe^{-t/\tau_q},
\label{eq:pronyseries}
\end{equation}
with $N_\varphi\in\mathbb{N}$, positive delay times $\{\tau_q\}^{N_\varphi}_{q=1}$ and 
positive coefficients $\{\varphi_q\}^{N_\varphi}_{q=0}$ which we can assume to be 
normalised so that $\varphi(0)=1$.

The form \eqref{eq:pronyseries} permits us to deal with the Volterra integrals in 
\eqref{eq:dispconlaw} and \eqref{eq:veloconlaw} in a way that avoids any reference to the past
`history' of the solution. 
In this approach, detailed later in \eqref{stress:p1}, \eqref{iv:p1}, \eqref{ode:p1},
\eqref{iv:p2} and \eqref{ode:p2}, the stress $\ushort{\sigma}$ is defined by using,
in place of the Volterra integral, `hidden' or \emph{internal variables}
(e.g.~\cite{VE,johnson}) that
evolve following an ordinary differential equation. Each of \eqref{eq:dispconlaw} and \eqref{eq:veloconlaw} gives rise to different internal variables and so will be considered
separately to give two formulations: the \emph{displacement} form and the \emph{velocity} form.
For each of these we will approximate the solution and the internal variables using the
standard continuous Galerkin Finite Element Method (CGFEM) to discretise
in space, and a second order implicit Crank-Nicolson finite difference method for the
time discretisation.

The plan of the paper is as follows.
In Section~\ref{sec:weakforms} we introduce the displacement form of the problem in Subsection~\ref{subsec:dispform}, and the velocity form in Subsection~\ref{subsec:veloform}.
The fully discrete approximations are then given in Sections~\ref{sec:discdispform}
(displacement form) and~\ref{sec:discveloform} (velocity form)
where we prove stability and \textit{a priori} error estimates. Gr{\"o}nwall's
lemma (e.g.~\cite{gronwall1919note,holte2009discrete}) is not used for these proofs and so
the constants in these bounds do not grow exponentially with time and
we can, therefore, have confidence in these schemes for long-time integration.
The Gr{\"o}nwall lemma allows for some simplification of analysis (see e.g. \cite{DGV,DG,thomee1984galerkin}), and to circumvent it requires some effort.
Here we rely on assumptions that arise from the physical character and properties
of the problem,
and then, to achieve the sharper bounds, our proofs use some long and technical calculations and details. These are sometimes omitted or just sketched out where
it aids presentation of the main ideas and specific arguments.
In Section~\ref{sec:numexp}
we use the FEniCS environment (see \cite{alnaes2015fenics},
\texttt{https://fenicsproject.org/})
to give the results of some numerical experiments, and explain how our software can be
acquired and the results reproduced. We finish in Section~\ref{sec:concs} with some general
comments. We also note that the results herein are presented in expanded form
in \cite{phdthesis}.

In terms of context we note that for the integral form of the quasistatic
(i.e.~where $\rho\ddot{u}$ is neglected) version of this viscoelasticity problem 
estimates that avoid Gr\"onwall's inequality were given in
\cite{SWWb,SWf,RiviereShawWheelerWhiteman03,SWj}.
For the dynamic problem we refer to DG-in-time, and DG-in-space methods in
\cite{Shaw11i,SWg} --- only the first of these avoided the Gr\"onwall lemma, whereas
the second is similar to what we refer to as the `displacement formulation' below.
Of these, \cite{Shaw11i,SWj,SWg} used space-time finite element formulations.
More generally, the integral form of the dynamic problem has been studied widely
in, for example, \cite{GYFa,PTWa}. The contribution of this paper is to give two
formulations of the dynamic problem using internal variables, and give
accompanying stability and error estimates that completely avoid the Gr\"onwall inequality.
As far as we are aware equivalent analyses are not currently available.
In terms of the well-posedness of problems of this type we refer to the well-known
paper \cite{Dafermos70}.

We introduce and use some standard notations so that $L_p(\Omega),H^s(\Omega)$ and $W^s_p(\Omega)$ denote the usual Lebesgue, Hilbert and Sobolev spaces. For any Banach space $X$, $\lVert\cdot\rVert_X$ is the $X$ norm. For example, $\llnorm{\cdot}$ is the $L_2$ norm induced by the $L_2$ inner product which we denote by $(\cdot,\cdot)$ for the entire domain but for $S\subset\bar\Omega$, $(\cdot,\cdot)_{L_2(S)}$ is the $L_2$ inner product over $S$. In the case of time dependent functions, we expand this notation such that if $f\in L_p(0,T;X)$ for some Banach space $X$, we define
$\lVert f\rVert_{L_p(0,t_0;X)}$ to be the $L_p(0,t_0)$ norm of 
$\lVert f(t)\rVert_{X}$.
% \[
% \lVert f\rVert_{L_p(0,t_0;X)}=\left(\int^{t_0}_0\lVert f(t)\rVert_X^p \, dt\right)^{1/p}
% \]
% for $t_0\leq T$ and $1\leq p<\infty$. When $p=\infty$, we shall use `essential supremum' norm where $\lVert f\rVert_{L_\infty(0,t_0;X)}=\mathop{\mathrm{ess~sup}}\limits_{0\leq t\leq t_0}\lVert f(t)\rVert_X $.
We also use the same notation for vector valued functions in Section~\ref{sec:concs}.
Lastly in this section, and for use later, we recall the trace inequality,
\begin{align}
\lVert v\rVert_{L_2(\partial\Omega)}&\leq C\hnorm {v} \textrm{ for any }v\in H^1(\Omega),
\label{cgtrace}
\end{align}
where $C$ is a positive constant depending only on $\Omega$ and its boundary $\partial\Omega$.

\section{Weak formulations}\label{sec:weakforms}\noindent
Our first step is to define the test space $V$,
\begin{align*}
V=\left\{v\in H^1(\Omega)\ |\ v=0\ \textrm{ on } {\Gamma_D}\right\},
\end{align*}
and then, multiplying \eqref{eq:momeq} by $v\in V$, integrating
by parts and using the boundary data gives, in a standard way, that
\begin{align}
\Lnorm{\rho\ddot u(t)}{v}
+(\sigma(t),\nabla v)
=F_d(t;v)\label{weakform}
\end{align}
for all $v\in V$ where the time dependent linear form $F_d$ is defined by
\begin{equation}
F_d(t;v)=\Lnorm{f(t)}{v}+(g_N(t),v)_{L_2(\Gamma_N)}.
\label{eq:linear_form}
\end{equation}
We now need to substitute for the stress using either the displacement or velocity forms.

\subsection{Displacement form}\label{subsec:dispform}\noindent
Recalling \eqref{eq:dispconlaw} and \eqref{eq:pronyseries} we write
\begin{eqnarray}
\ushort\sigma(t)=D\nabla \left(u(t)-\sum_{{q}=1}^{N_\varphi}\psi_q(t)\right),
\label{stress:p1}
\end{eqnarray}
where, for $1\leq {q}\leq N_\varphi$, the internal variables are defined by
\begin{eqnarray}
\psi_{q}(t):=\frac{\varphi_{q}}{\tau_{q}}\int^t_0e^{-(t-s)/\tau_{q}}u(s)\ ds,
\label{iv:p1}
\end{eqnarray}
and satisfy the following ODEs,
\begin{align}\label{ode:p1}
\dot\psi_{q}(t)=\frac{\varphi_{q}}{\tau_{q}}u(t)-\frac{1}{\tau_{q}}\psi_{q}(t)\qquad\textrm{for }q=1,2,\ldots,N_\varphi,
\end{align}
with $\psi_q(0)=0$. Our weak formulation \eqref{weakform} can now be written as
\begin{align}
\Lnorm{\rho\ddot u(t)}{v}+a(u(t),v)-\sum_{q=1}^{N_\varphi}a(\psi_{q}(t),v)
=F_d(t;v)
\qquad\forall v\in V.
\label{weakp1}
\end{align}
In this the symmetric bilinear form $a\colon V\times V\to\mathbb{R}$ is defined by
$a(w,v)=\Lnorm{D\nabla w} {\nabla v}$ and is easily shown to be continuous on $V$. 
Moreover, it follows from our assumption on $\Gamma_D$ that it is also coercive on $V$,
\cite{graser2015note}, and so the \emph{energy norm} defined by $\Vert v\Vert_V := \sqrt{a(v,v)}$,
for $v\in V$, satisfies $\kappa\hnorm v^2\leq\enorm v^2\leq D\hnorm v^2$ for a positive constant
$\kappa$. Thus $(V,a(\cdot,\cdot))$ is a Hilbert space equivalent
to $(H^1(\Omega),(\cdot,\cdot)_{H^1(\Omega)})$. 

We use this bilinear form, or energy inner product, to enforce each internal variable ODE,
and then arrive at the following weak problem.

\textbf{(P1)}
Find $(0,T]\to V$ maps $u$, $\psi_{1}$, $\psi_{2}, \ldots, \psi_{N_\varphi}$ such that
\begin{alignat}{2}
\Lnorm{\rho\ddot u(t)}{v}+a(u(t),v)-\sum_{q=1}^{N_\varphi}a(\psi_{q}(t),v)
&=F_d(t;v)\qquad
&&\forall v\in V,\label{weak:p1e1}
\\
\tau_{q}a(\dot\psi_{q}(t),v)+a(\psi_{q}(t),v)
&=\varphi_{q}a(u(t),v) \qquad
&&\forall v\in V,\ q=1,\ldots,N_\varphi \label{weak:p1e2}
\end{alignat}
with $u(0)=u_0,$ $\dot u(0)=w_0$ and $\psi_{q}(0)=0, \ \forall {q}\in\{1,\ldots,N_\varphi\}$.

\subsection{Velocity form}\label{subsec:veloform}\noindent
On the other hand, using \eqref{eq:veloconlaw} and \eqref{eq:pronyseries}
with the velocity form of internal variable given by
\begin{equation}
\zeta_{q}(t)=\int^t_0\varphi_{q}e^{-(t-s)/\tau_{q}} \dot u(s)\ ds,
\label{iv:p2}
\end{equation}
for each $q=1,\ldots,N_\varphi$, we have
\begin{equation}
\dot\zeta_{q}(t)+\frac{1}{\tau_{q}}\zeta_{q}(t) =\varphi_{q}\dot u(t)\qquad\textrm{for }q=1,2,\ldots,N_\varphi,
\label{ode:p2}
\end{equation}
with $\zeta_{q}(0)=0$. Noticing that
$\psi_{q}(t)=\varphi_{q} u(t)-\varphi_{q}e^{-t/\tau_{q}}u_0-\zeta_{q}(t)$
(integrate by parts)
and recalling that $\varphi(0) = 1$ we can observe that 
\[
u(t)-\sum_{{q}=1}^{N_\varphi}\psi_{q}(t)
=
\varphi_0 u(t)+
\sum_{{q}=1}^{N_\varphi}\Big(
\varphi_qe^{-t/\tau_{q}}u_0+\zeta_{q}(t)
\Big).
\]
Using this in \eqref{stress:p1}, substituting the result into \eqref{weakform}
and incorporating \eqref{ode:p2}, gives the weak problem for the velocity form.

\textbf{(P2)} Find $(0,T]\to V$ maps $u$, $\zeta_{1}$, $\zeta_{2},\ldots, \zeta_{N_\varphi}$
such that
\begin{alignat}{1}
\Lnorm{\rho\ddot u(t)}{v}+\varphi_0a(u(t),v)+\sum_{q=1}^{N_\varphi}a(\zeta_{q}(t),v)
=F_v(t;v) \qquad
&\forall v\in V,\label{weak:p2e1}
\\
\tau_{q}a(\dot\zeta_{q}(t),v)+a(\zeta_{q}(t),v)
=\tau_q\varphi_{q}a(\dot u(t),v)\qquad
&\forall v\in V,\ q=1,\ldots, N_\varphi,\label{weak:p2e2}
\end{alignat}
with $u(0)=u_0,$ $\dot u(0)=w_0$, $\zeta_{q}(0)=0, \ \forall {q}$
and $F_v(t;v)= F_d(t;v)
-\sum\limits_{{q}=1}^{N_\varphi}\varphi_qe^{-t/\tau_{q}}a(u_0,v)$.

We now move to the fully discrete schemes for \textbf{(P1)} in the next section
and for \textbf{(P2)} in Section~\ref{sec:discveloform}.

%\section{Fully discrete formulation}\label{sec:discforms}\noindent
\section{Fully discrete formulation: Displacement form}
\label{sec:discdispform}\noindent
Let $V^h \subset V$ be a conforming finite element space built with 
continuous piecewise Lagrange basis functions with respect to an underlying 
quasi-uniform mesh with mesh-size characterised by $h$.
We write $t_n=n\Delta t$ with time step $\Delta t = T/N$ for
$N\in\mathbb{N}$, and denote the fully discrete approximations to
$u$ and $\dot{u}$ by $u(t_n)=u^n\approx Z_h^n\in V^h$ and
$\dot{u}(t_n)=\dot{u}^n\approx W_h^n\in V^h$. 
Furthermore, we will use the following approximations, 
\[
\frac{\ddot u(t_{n+1})+\ddot u(t_{n})}{2}\approx\frac{W^{n+1}_h-W^{n}_h}{\Delta t}
\qquad\textrm{ and }\qquad
\frac{ u(t_{n+1})+ u(t_{n})}{2}\approx\frac{Z^{n+1}_h+Z^n_h}{2},
\]
and will impose the relation
\begin{equation}\label{r1}
  \frac{W^{n+1}_h+W^{n}_h}{2}=\frac{Z^{n+1}_h-Z^n_h}{\Delta t}  
\end{equation}
in our fully discrete schemes.

Our fully discrete formulation for \textbf{(P1)} is:

$\mathrm{(\mathbf{P1})}^h$
Find
$Z^n_h$, $W^n_h$, ${\Psi}_{h1}^n$, ${\Psi}_{h2}^n, \ldots,$${\Psi}_{hN_\varphi}^n\in V^h$
for
$n=0,\ldots,N$ such that \eqref{r1} holds along with:
\begin{align}
\Lnorm{\rho\frac{W^{n+1}_h-W^{n}_h}{\Delta t}}{v}+a\left(\frac{Z_h^{n+1}+Z_h^{n}}{2},v\right)
- &\sum\limits_{q=1}^{N_\varphi}a\left(\frac{{\Psi}_{hq}^{n+1}+{\Psi}_{hq}^{n}}{2},v\right)
\nonumber\\
& =\frac{F_d(t_{n+1};v)+F_d(t_{n};v)}{2},\label{p1f1}
\\
\tau_{q}a\left(\frac{{\Psi}_{hq}^{n+1}-{\Psi}_{hq}^{n}}{\Delta t},v\right)+a\left(\frac{{\Psi}_{hq}^{n+1}+{\Psi}_{hq}^{n}}{2},v\right)
& =\varphi_qa\left(\frac{Z_h^{n+1}+Z_h^{n}}{2},v\right) \textrm{ for each }{q},
\label{p1f2}
\\
a(Z_h^0,{v})
& = a({u_0},{v}) ,\label{p1f3}\\
\Lnorm{W^0_h}{v} & =\Lnorm{w_0}{v},\label{p1f4}\\
{\Psi}_{hq}^0
& =0\qquad\textrm{for each } q,
\label{p1f5}
\end{align}
each for all $v\in V^h$.
It follows immediately that 
$\enorm{Z^0_h}\leq\enorm{u_0}$ and $\textrm{and}\llnorm{W^0_h}\leq\llnorm{w_0}$, and we 
have the following stability estimate.

% When we put global basis functions for $v$, the resulting linear system from \eqref{p1f1}-\eqref{p1f4} has the existence and uniqueness of the solution by the stability bounds.
% Note that \eqref{p1f3} and \eqref{p1f4} imply
% \begin{align*}
%     \enorm{Z^0_h}^2=a(Z^0_h,Z^0_h)=a(u_0,Z^0_h)\leq\enorm{Z^0_h}\enorm{u_0},
% \end{align*}
% and
% \begin{align*}
%     \llnorm{W^0_h}^2=\Lnorm{W^0_h}{W^0_h}=\Lnorm{w_0}{W^0_h}\leq\llnorm{W^0_h}\llnorm{w_0},
% \end{align*}
% so that we have
% \begin{align*}
%     \enorm{Z^0_h}\leq\enorm{u_0},&&\textrm{and}&&\llnorm{W^0_h}\leq\llnorm{w_0}.
% \end{align*}

\begin{theorem}\label{thm:fully:stability:p1} 
Suppose $f\in C(0,T;L_2(\Omega))$, $g_N\in H^1(0,T;L_2(\Gamma_N))\cap C(0,T;L_2(\Gamma_N))$ and \linebreak $u_0\in H^1(\Omega)$, then	
$\mathrm{(\mathbf{P1})}^h$ has a unique
solution. Moreover, there exists a positive constant $C$ depending on 
$\Omega,\ \partial\Omega$ the sets $\{\varphi_q\}_{q=0}^{N_\varphi}$ and $\{\tau_q\}_{q=1}^{N_\varphi}$, but independent of $h$, $\Delta t$, $T$ and
the exact and numerical solutions, such that
\begin{multline*}
\frac{\rho}{2}\max_{0\leq n\leq N}\llnorm{W^{n}_h}^2+\frac{\varphi_0}{4}\max_{0\leq n\leq N}\enorm{Z^{n}_h}^2
\\
+\sum_{q=1}^{N_\varphi}\frac{\varphi_0}{2\varphi_q^2N_\varphi+\varphi_q\varphi_0}\max_{0\leq n\leq N}\enorm{{\Psi}^{n}_{hq}}^2
+\sum_{q=1}^{N_\varphi}\sum_{n=0}^{N-1}\frac{2\tau_q}{\Delta t\varphi_q}\enorm{{\Psi}^{n+1}_{hq}-{\Psi}^{n}_{hq}}^2
\\ \leq
CT^2\bigg(\llnorm{w_0}^2+\enorm{u_0}^2+\ilnorm{f}^2+\hgnorm{g_N}^2\bigg).
\end{multline*}
% %
% % original version
% \begin{align*}
% \frac{\rho}{2}\max_{0\leq n\leq N}\llnorm{W^{n}_h}^2+\frac{\varphi_0}{4}\max_{0\leq n\leq N}\enorm{Z^{n}_h}^2
% &+\sum_{q=1}^{N_\varphi}\frac{\varphi_0}{2\varphi_q^2N_\varphi+\varphi_q\varphi_0}\max_{0\leq n\leq N}\enorm{{\Psi}^{n}_{hq}}^2
% \\
% +\sum_{q=1}^{N_\varphi}\sum_{n=0}^{N-1}\frac{2\tau_q}{\Delta t\varphi_q}\enorm{{\Psi}^{n+1}_{hq}-{\Psi}^{n}_{hq}}^2
% & \leq
% CT^2\bigg(\llnorm{w_0}^2+\enorm{u_0}^2+\ilnorm{f}^2
% \\
% &\qquad\qquad
% +\lgnorm{\dot g_N}^2+\ignorm{g_N}^2\bigg).
% \end{align*}
\end{theorem}

\begin{proof}
The existence and uniqueness follows from the stated bound, so we have
only to establish that.
Choose $m\in\mathbb{N}$ such that $1\leq m\leq N$.
Put $v=\Delta t(W^{n+1}_h+W^n_h)$ for $0\leq n\leq m-1$ into \eqref{p1f1},
then $v=2({\Psi}^{n+1}_{hq}-{\Psi}^n_{hq})/\varphi_q$
into \eqref{p1f2} and sum over $q$. Then add all these results and sum over
$n=0$ to $n=m-1$. After noting that,
\begin{align*}
\frac{\Delta t}{2}&\Enorm{\Psi_{hq}^{n+1}+\Psi_{hq}^n}{W^{n+1}_h+W^n_h}+\Enorm{\Psi_{hq}^{n+1}-\Psi_{hq}^n}{Z^{n+1}_h+Z^n_h}
\\
& = 2\Enorm{\Psi_{hq}^{n+1}}{Z_h^{n+1}}-2\Enorm{\Psi_{hq}^n}{Z_h^n},
\end{align*}
we obtain,
\begin{align}
{\rho}&\llnorm{W^{m}_h}^2+\enorm{Z^{m}_h}^2+\sum_{q=1}^{N_\varphi}\frac{1}{\varphi_q}\enorm{{\Psi}^{m}_{hq}}^2+\sum_{q=1}^{N_\varphi}\sum_{n=0}^{m-1}\frac{2\tau_q}{\Delta t\varphi_q}\enorm{{\Psi}^{n+1}_{hq}-{\Psi}^{n}_{hq}}^2
\nonumber\\
=&\rho\llnorm{W^0_h}^2+\enorm{Z^{0}_h}^2+\frac{\Delta t}{2}\sum_{n=0}^{m-1}\bigg(F_d\left(t_{n+1};W^{n+1}_h+W^n_h\right)+F_d\left(t_{n};W^{n+1}_h+W^n_h\right)\bigg)
\nonumber\\
&+\sum_{q=1}^{N_\varphi}2a(Z^{m}_h,{\Psi}^{m}_{hq}),
\label{eq:P1stabproof_1}
\end{align}
We first consider the third term on the right where, from
\eqref{r1} and the definition, \eqref{eq:linear_form}, 
\begin{align*}
    \frac{\Delta t}{2}&\sum_{n=0}^{m-1}\bigg(
    F_d(t_{n+1};W^{n+1}_h+W^n_h)+F_d(t_{n};W^{n+1}_h+W^n_h)
    \bigg)
    \\
    =&\frac{\Delta t}{2}\sum_{n=0}^{m-1}\Lnorm{f(t_{n+1})+f(t_{n})}{W^{n+1}_h+W^n_h}+\sum_{n=0}^{m-1}\left(g_N(t_{n+1})+g_N(t_{n}),Z^{n+1}_h-Z^{n}_h\right)_{L_2(\Gamma_N)}.
\end{align*}
We sum by parts in the second term to introduce the difference 
$g_N(t_{n+1})-g_N(t_{n})$, and replace this with the integral of $\dot{g}_N$ over
the time step. We estimate the remaining terms in a standard way using
the trace, Cauchy Schwarz and Young's inequalities
for positive $\epsilon_a$ $\epsilon_b$ and $\epsilon_c$,
and it follows that
\begin{align}
\bigg|\frac{\Delta t}{2}&\sum_{n=0}^{m-1}(F_d(t_{n+1};W^{n+1}_h+W^n_h)+F_d(t_{n};W^{n+1}_h+W^n_h))\bigg|
\nonumber\\
\nonumber\leq&\frac{2\Delta t}{\epsilon_a}\sum_{n=0}^{N}\llnorm{f(t_{n})}^2+{2(T+\Delta t)\epsilon_a}\max_{0\leq n\leq N}\llnorm{W^{n}_h}^2\\
&\nonumber+\bigg(\frac{C}{\epsilon_b}+1\bigg)\max_{0\leq n\leq N}\gnorm{g_N(t_{n})}^2+{C\epsilon_b}\max_{0\leq n\leq N}\enorm{Z^{n}_h}^2+{C}\enorm{Z^{0}_h}^2\\
&+\frac{C}{\epsilon_c}\lgnorm{\dot g_N}^2+{2C(T+\Delta t)\epsilon_c}\max_{0\leq n\leq N}\enorm{Z^{n}_h}^2\label{pf:sta:p1:eq2}.
\end{align}
Secondly, for the fourth term on the right of \eqref{eq:P1stabproof_1},
we get from the Cauchy-Schwarz and Young's inequalities that 
\[
    \sum_{q=1}^{N_\varphi}2a(Z^{m}_h,{\Psi}^{m}_{hq})
    \leq\sum_{q=1}^{N_\varphi}\epsilon_q\enorm{Z^m_h}^2+\sum_{q=1}^{N_\varphi}\frac{1}{\epsilon_q}\enorm{{\Psi}^m_{hq}}^2,
\]
for any $\epsilon_q>0$ for each $q$. 
Returning to \eqref{eq:P1stabproof_1} and using 
these estimates results in
\begin{align*}
{\rho}&\llnorm{W^{m}_h}^2
+\left(1-\sum_{q=1}^{N_\varphi}\epsilon_q\right)
\enorm{Z^{m}_h}^2
+\sum_{q=1}^{N_\varphi}
\left(\frac{1}{\varphi_q}-\frac{1}{\epsilon_q}\right)
\enorm{{\Psi}^{m}_{hq}}^2
\\
& +\sum_{q=1}^{N_\varphi}\sum_{n=0}^{m-1}
\frac{2\tau_q}{\Delta t\varphi_q}\enorm{{\Psi}^{n+1}_{hq}-{\Psi}^{n}_{hq}}^2
\leq
\rho\llnorm{W^0_h}^2+(1+C)\enorm{Z^{0}_h}^2
+\frac{2\Delta t}{\epsilon_a}\sum_{n=0}^{N}\llnorm{f(t_{n})}^2
\\
&+{2(T+\Delta t)\epsilon_a}\max_{0\leq n\leq N}\llnorm{W^{n}_h}^2
+\bigg(\frac{C}{\epsilon_b}+1\bigg)\max_{0\leq n\leq N}\gnorm{g_N(t_{n})}^2+{C\epsilon_b}\max_{0\leq n\leq N}\enorm{Z^{n}_h}^2
\\
&+\frac{C}{\epsilon_c}\lgnorm{\dot g_N}^2+{2C(T+\Delta t)\epsilon_c}\max_{0\leq n\leq N}\enorm{Z^{n}_h}^2.
\end{align*}
Next, we recall our sign assumptions on the coefficients in \eqref{eq:pronyseries} and take
$\epsilon_q=\varphi_q+\varphi_0/(2N_\varphi)>0$ for each $q$. Recalling also that
$\varphi(0)=1$ we note that
% \[
% 1-\sum_{q=1}^{N_\varphi}\epsilon_q=\frac{\varphi_0}{2}>0
% \textrm{ and }
% \frac{1}{\varphi_q}-\frac{1}{\epsilon_q}=\frac{\varphi_0}{2\varphi_q^2N_\varphi+\varphi_q\varphi_0}>0,\ \forall{q}\in\{1,\ldots,N_\varphi\},
% \]
% 
\begin{align*}
1-\sum_{q=1}^{N_\varphi}\epsilon_q
& =
1-\sum_{q=1}^{N_\varphi}\varphi_q-\sum_{q=1}^{N_\varphi}\frac{\varphi_0}{2N_\varphi}
= 1 - (1-\varphi_0) - \frac{\varphi_0}{2}
= \frac{\varphi_0}{2}>0
\\
\text{and }\qquad
\frac{1}{\varphi_q}-\frac{1}{\epsilon_q}
& =
\frac{\varphi_0}{2\varphi_q^2N_\varphi+\varphi_q\varphi_0}>0
\qquad \forall{q}\in\{1,\ldots,N_\varphi\},
\end{align*}
to obtain,
\begin{align*}
{\rho}&\llnorm{W^{m}_h}^2
+\frac{\varphi_0}{2}\enorm{Z^{m}_h}^2
+\sum_{q=1}^{N_\varphi}\frac{\varphi_0}{2\varphi_q^2N_\varphi
+\varphi_q\varphi_0}\enorm{{\Psi}^{m}_{hq}}^2
+\sum_{q=1}^{N_\varphi}\sum_{n=0}^{m-1}
\frac{2\tau_q}{\Delta t\varphi_q}\enorm{{\Psi}^{n+1}_{hq}-{\Psi}^{n}_{hq}}^2
\\
\leq&\rho\llnorm{W^0_h}^2+(1+C)\enorm{Z^{0}_h}^2+\frac{2\Delta t}{\epsilon_a}\sum_{n=0}^{N}\llnorm{f(t_{n})}^2+{2(T+\Delta t)\epsilon_a}\max_{0\leq n\leq N}\llnorm{W^{n}_h}^2\\
&+\bigg(\frac{C}{\epsilon_b}+1\bigg)\max_{0\leq n\leq N}\gnorm{g_N(t_{n})}^2+{C\epsilon_b}\max_{0\leq n\leq N}\enorm{Z^{n}_h}^2\\
&+\frac{C}{\epsilon_c}\lgnorm{\dot g_N}^2+{2C(T+\Delta t)\epsilon_c}\max_{0\leq n\leq N}\enorm{Z^{n}_h}^2.
\end{align*}
From this we can get,
\begin{align*}
{\rho}&\max_{0\leq n\leq N}\llnorm{W^{n}_h}^2+\frac{\varphi_0}{2}\max_{0\leq n\leq N}\enorm{Z^{n}_h}^2+\sum_{q=1}^{N_\varphi}\frac{\varphi_0}{2\varphi_q^2N_\varphi+\varphi_q\varphi_0}\enorm{{\Psi}^{m}_{hq}}^2\\&+\sum_{q=1}^{N_\varphi}\sum_{n=0}^{m-1}\frac{2\tau_q}{\Delta t\varphi_q}\enorm{{\Psi}^{n+1}_{hq}-{\Psi}^{n}_{hq}}^2
\\
\leq&
3\bigg(\rho\llnorm{W^0_h}^2+(1+C)\enorm{Z^{0}_h}^2+\frac{2\Delta t}{\epsilon_a}\sum_{n=0}^{N}\llnorm{f(t_{n})}^2
\\
&+{2(T+\Delta t)\epsilon_a}\max_{0\leq n\leq N}\llnorm{W^{n}_h}^2+\bigg(\frac{C}{\epsilon_b}+1\bigg)\max_{0\leq n\leq N}\gnorm{g_N(t_{n})}^2\\
&+{C\epsilon_b}\max_{0\leq n\leq N}\enorm{Z^{n}_h}^2+\frac{C}{\epsilon_c}\lgnorm{\dot g_N}^2+{2C(T+\Delta t)\epsilon_c}\max_{0\leq n\leq N}\enorm{Z^{n}_h}^2\bigg),
\end{align*}
and then choosing $\epsilon_a=\rho/(12(T+\Delta t))$, $\epsilon_b=\varphi_0/(24C)$ and $\epsilon_c=\varphi_0/(48C(T+\Delta t))$
and recalling that $\llnorm{W_h^0}\leq\llnorm{w_0}$ and 
$\enorm{Z_h^0}\leq\enorm{u_0}$,
we conclude that there is a positive constant $C$ such that
\begin{multline*}
\frac{\rho}{2}\max_{0\leq n\leq N}\llnorm{W^{n}_h}^2
+\frac{\varphi_0}{4}\max_{0\leq n\leq N}\enorm{Z^{n}_h}^2
+\sum_{q=1}^{N_\varphi}
\frac{\varphi_0}{2\varphi_q^2N_\varphi+\varphi_q\varphi_0}\enorm{{\Psi}^{m}_{hq}}^2
\\
+\sum_{q=1}^{N_\varphi}\sum_{n=0}^{m-1}\frac{2\tau_q}{\Delta t\varphi_q}\enorm{{\Psi}^{n+1}_{hq}-{\Psi}^{n}_{hq}}^2
\leq CT\bigg(\llnorm{w_0}^2+\enorm{u_0}^2
\\
+{\Delta t}\sum_{n=0}^{N}\llnorm{f(t_{n})}^2+\lgnorm{\dot g_N}^2
+\max_{0\leq n\leq N}\gnorm{g_N(t_{n})}^2\bigg),
\end{multline*}
for ${0\leq m\leq N-1}$. As $T+\Delta t\leq2T$ the constant $C$ is independent of $h$,
$\Delta t$, $T$ and the exact and numerical solutions, but it depends on the
trace inequality constant and the physical quantities of $\rho$, $D$ and the
internal variables.
Noting that $m$ is arbitrary, recalling the hypotheses on $f$ and $g$,
% 
% Since $f$ and $g$ are continuous and bounded, we have
% \[{\Delta t}\sum_{n=0}^{N}\llnorm{f(t_{n})}^2\leq 2T\ilnorm{f}\text{and }\max_{0\leq n\leq N}\gnorm{g_N(t_{n})}^2\leq \ignorm{g_N}^2.\]
% Recalling that $m$ is arbitrary
% 
and that $\ignorm{g_N}\le C\hgnorm{g_N}$, then completes the proof.
\end{proof}

Notice that the proof of Theorem~\ref{thm:fully:stability:p1} is an
example of how the $e^T$ dependence of the constant that would arise
from using Gr\"onwall's inequality can be avoided for these
viscoelasticity problems. It is noticeable that the proof was rather more involved than would be needed if we used Gr\"onwall. The proofs that 
follow exhibit a similar amount of additional effort.

% , we can observe that a discrete solution is bounded by only data terms. However, if we use Gr\"onwall's inequality, the solution is bounded by $\exp(T)$ instead of $T^2$, which giving poor stability for large $T$. 

% From the stability bound in Theorem 3.1, the fully discrete solutions are bounded by the data such as boundary conditions, initial conditions and source terms. It means that if the data is given by zero data, the solutions must be zeros. 

% Recall the concept of linear algebra,
% \[A\ushort x=\ushort{b}\textrm{ is solved uniquely}\Longleftrightarrow A\ushort x=\ushort{0}\textrm{ only if }\ushort{x}=\ushort{0}.\]
% Note that solving the fully discrete formulation is equivalent to solving the resulting linear system. In the above sense, $\ushort{x}$ represents our solution and $\ushort{b}$ is defined by the given data. Therefore, if the data is given by 0 then the solution should be also zero so that the linear system is solved uniquely.

We now turn to error bounds and begin by defining the elliptic projection $R:V\mapsto V^h$
(e.g. \cite{wheeler}) for $w\in V$ by 
\[
a(Rw,v)=a(w,v),\ \forall v\in V^h,
\]
and note the resulting Galerkin orthogonality such that for any $w\in V$
$$a(w-Rw,v)=0,\ \forall v\in V^h.$$
It follows that $\frac{\partial}{\partial t}(Rw) = R\dot{w}$ if $\dot{w}\in V$. We will make use of the following standard result.

\begin{lemma}\label{lemma:estimate:elliptic}
\textnormal{(see \cite[pg 731-732]{wheeler} and \cite{MTE})}
If $w\in  H^{s_2}(\Omega)\subset V$ and $V^h$ uses
piecewise polynomials of degree $s_1$ in $V$, then
\[
\llnorm{w-Rw}\leq C|w|_{H^r(\Omega)}h^{r-1}
\textrm{ and }\enorm{w-Rw}\leq C|w|_{H^r(\Omega)}h^{r-1}
\]
for some positive constant $C$, and where $r:=\min{(s_1+1,s_2)}$. Furthermore, if there is elliptic regularity, we have
$\llnorm{w-Rw}\leq Ch^{r}|w|_{H^r(\Omega)}$. 
\end{lemma}
By `elliptic regularity' here we of course mean that, for a constant $C$,
\begin{equation}
\lvert\xi\rvert_{H^2(\Omega)}
\leq C\lVert \Delta\xi\rVert_{L_2(\Omega)},
\label{eq:ellipticRegularity:H2}
\end{equation}
where, following the usual Aubin-Nitsche duality argument, $\xi$ solves an associated
dual problem. This property is known to hold for $\Omega$ either a smooth domain or a convex polytope, and with $\Gamma_N=\emptyset$. See, for example, 
\cite{dauge1988elliptic,grisvard2011elliptic,MTE} where more general boundary conditions are
also considered.

% When the domain is neither smooth nor convex, it fails to get higher order with respect to $L_2$ norm \cite{MTE}. Hence either smoothness or convexity on the domain is an essential condition of elliptic regularity. Moreover, with $\Gamma_N=\emptyset$, elliptic regularity holds for smooth $\Omega$ and convex polytope $\Omega$ e.g. see \cite{dauge1988elliptic,grisvard2011elliptic,MTE}.
% \begin{remark}
% 	Elliptic regularity is required for optimal $L_2(\Omega)$ error estimates. The general case for elliptic regularity estimates is given by
% 	\begin{equation}
% 	\lVert v\rVert_{W^2_p(\Omega)}\leq\lVert \Delta v\rVert_{L_p(\Omega)},\qquad1<p<\mu,\label{eq:ellipticRegularity}
% 	\end{equation}where $\mu$ depends on $\partial\Omega$ e.g. see \cite{MTE,dauge1988elliptic,grisvard2011elliptic}.  Therefore, if we assume that elliptic regularity is satisfied, it holds that 
% 	\begin{equation}
% 	\lvert v\rvert_{H^2(\Omega)}\leq\lVert \Delta v\rVert_{L_2(\Omega)},\label{eq:ellipticRegularity:H2}
% 	\end{equation}
% 	then we have $\llnorm{w-Rw}\leq C|w|_{H^r(\Omega)}h^{r}$.
% \end{remark}

The approach here is standard in that we split the error using the
elliptic projection. To this end, let
\begin{gather*}
    \theta:=u-Ru,\qquad    \chi^n:=Z_h^n-Ru^n,\qquad    \varpi^n:=W_h^n- R\dot u^n,\\
    \vartheta_q:=\psi_q-R\psi_q,\qquad    \varsigma^n_q:={\Psi}_{hq}^n-R\psi_q^n,
\end{gather*}for each $q$. Additionally, we define
\begin{align*}
    e_h^n:=u^n-Z^n_h=\theta^n-\chi^n,\qquad
    \tilde e_h^n:=\dot u^n-W^n_h=\dot\theta^n-\varpi^n,
\end{align*} and also $\psi^n-\Psi^n_{hq}=\vartheta^n_q-\varsigma^n_q$ for each $q$.
\begin{lemma}\label{lemma:estimate:p1}
Suppose $u\in H^4(0,T;H^{s_2}(\Omega))\cap C^1(0,T;H^{s_2}(\Omega))$ then,
using \textnormal{Lemma~\ref{lemma:estimate:elliptic}},
\begin{align*}
        \max_{0\leq k\leq N}\llnorm{\varpi^{k}}&
    +\max_{0\leq k\leq N}\enorm{\chi^{k}}+
    \sum_{q=1}^{N_\varphi}\max_{0\leq k\leq N}\enorm{\varsigma_q^{k}}+\sqrt{{\Delta t}\sum_{n=0}^{N-1}\sum_{q=1}^{N_\varphi}\frac{\tau_q}{\varphi_q}\enorm{\frac{\varsigma_q^{n+1}-\varsigma_q^n}{\Delta t}}^2}
    \\\nonumber
\leq& CT\lVert u\rVert_{H^4(0,T;H^{s_2}(\Omega))}(h^{r-1}+\Delta t^2).
\end{align*}
Furthermore, if we also assume elliptic regularity,
\begin{align*}
        \max_{0\leq k\leq N}&\llnorm{\varpi^{k}}
    +\max_{0\leq k\leq N}\enorm{\chi^{k}}+
    \sum_{q=1}^{N_\varphi}\max_{0\leq k\leq N}\enorm{\varsigma_q^{k}}+\sqrt{{\Delta t}\sum_{n=0}^{N-1}\sum_{q=1}^{N_\varphi}\frac{\tau_q}{\varphi_q}\enorm{\frac{\varsigma_q^{n+1}-\varsigma_q^n}{\Delta t}}^2}
    \\\nonumber
\leq& CT\lVert u\rVert_{H^4(0,T;H^{s_2}(\Omega))}(h^{r}+\Delta t^2).
\end{align*}
Here, $C$ is a positive constant that depends on
 $\Omega$, $\partial \Omega$ and the problem coefficients
 $\rho$, $D$, $\{\varphi_q\}_{q=0}^{N_\varphi}$ and $\{\tau_q\}_{q=1}^{N_\varphi}$, but is independent of $h$, $\Delta t$, $T$, and the exact and numerical solutions.
\end{lemma}

\begin{proof}
Averaging \eqref{weak:p1e1} at $t_{n+1}$ and $t_n$ and subtracting
the result from \eqref{p1f1} gives,
\begin{align*}
&\Lnorm{\frac{\rho}{2}(\ddot u^{n+1}+\ddot u^n)
-\frac{\rho}{\Delta t}(W_h^{n+1}-W_h^n)}{v}
+\frac{1}{2}\Enorm{(u^{n+1}+u^n)-(Z_h^{n+1}+Z_h^n)}{v}
\\
&\qquad
-\frac{1}{2}\sum_{q=1}^{N_\varphi}
\Enorm{(\psi_q^{n+1}+\psi_q^n)-({\Psi}_{hq}^{n+1}+{\Psi}_{hq}^n)}{v}
=0
\end{align*}
for any $v\in V^h$.
Using Galerkin orthogonality, we can rewrite this as
\begin{align}
\frac{\rho}{\Delta t}&\Lnorm{\varpi^{n+1}-\varpi^n}{v}
+\frac{1}{2}\Enorm{\chi^{n+1}+\chi^n}{v}
-\frac{1}{2}\sum_{q=1}^{N_\varphi}
\Enorm{\varsigma_q^{n+1}+\varsigma_q^n}{v}\nonumber
\\
\qquad &=
\frac{\rho}{\Delta t}\Lnorm{\dot\theta^{n+1}-\dot\theta^n}{v}+\rho\Lnorm{\mathcal{E}^n_1}{v},
\label{p1fe0}
\end{align}
for $0\leq n\leq N-1$, where
\[\mathcal{E}_1(t):=\frac{\ddot u{(t+\Delta t)}+\ddot u(t)}{2}-\frac{\dot u{(t+\Delta t)}-\dot u(t)}{\Delta t}.\]
Note that by \eqref{r1} we have,
\begin{equation}
    \frac{\chi^{n+1}-\chi^{n}}{\Delta t}=\frac{Z^{n+1}-Z^{n}}{\Delta t}
    -\frac{Ru^{n+1}-Ru^{n}}{\Delta t}
    =\frac{\varpi^{n+1}+\varpi^{n}}{2}-\mathcal{E}_2^n-\mathcal{E}_3^n
    \label{r2},
\end{equation}
where
\begin{align*}
\mathcal{E}_2(t)
& :=\frac{\dot \theta(t+\Delta t)+\dot \theta(t)}{2}
-\frac{\theta(t+\Delta t)-\theta(t)}{\Delta t},
\\
\mathcal{E}_3(t)
& :=\frac{u(t+\Delta t)-u(t)}{\Delta t}
-\frac{\dot u(t+\Delta t)+\dot u(t)}{2},
\end{align*}
and then choosing $v=\frac{\chi^{n+1}-\chi^{n}}{\Delta t}$
in \eqref{p1fe0}, and using \eqref{r2}, we can derive the following,
\begin{align}\label{p1fe1}
\nonumber
\frac{\rho}{2\Delta t}&\left(\llnorm{\varpi^{n+1}}^2-\llnorm{\varpi^n}^2\right)
+\frac{1}{2{\Delta t}}\left(\enorm{\chi^{n+1}}^2-\enorm{\chi^n}^2\right)
\\\nonumber&\qquad
-\frac{1}{2}\sum_{q=1}^{N_\varphi}\Enorm{\varsigma_q^{n+1}+\varsigma_q^n}{\frac{\chi^{n+1}-\chi^{n}}{\Delta t}}
\\\nonumber &
=\frac{\rho}{2\Delta t}\Lnorm{\dot\theta^{n+1}-\dot\theta^n}{\varpi^{n+1}+\varpi^n}-\frac{\rho}{\Delta t}\Lnorm{\dot\theta^{n+1}-\dot\theta^n}{\mathcal{E}_2^n}
\\\nonumber&\qquad
-\frac{\rho}{\Delta t}\Lnorm{\dot\theta^{n+1}-\dot\theta^n}{\mathcal{E}_3^n}+\frac{\rho}{2}\Lnorm{\mathcal{E}_1^n}{\varpi^{n+1}+\varpi^n}-{\rho}\Lnorm{\mathcal{E}_1^n}{\mathcal{E}_2^n}
\\&\qquad
-{\rho}\Lnorm{\mathcal{E}_1^n}{\mathcal{E}_3^n}+\frac{\rho}{\Delta t}\Lnorm{\varpi^{n+1}-\varpi^n}{\mathcal{E}_2^n}+\frac{\rho}{\Delta t}\Lnorm{\varpi^{n+1}-\varpi^n}{\mathcal{E}_3^n}.
\end{align}
Summing this over $n=0,\ldots,m-1$, where $m\leq N$, we get
\begin{align}\nonumber
\frac{\rho}{2\Delta t}&\llnorm{\varpi^{m}}^2
+\frac{1}{2{\Delta t}}\enorm{\chi^{m}}^2
    \nonumber
-\frac{1}{2{\Delta t}}\sum_{n=0}^{m-1}\sum_{q=1}^{N_\varphi}
\Enorm{\varsigma_q^{n+1}+\varsigma_q^n}{{\chi^{n+1}-\chi^{n}}}
\\\nonumber
=&\frac{\rho}{2\Delta t}\llnorm{\varpi^0}^2+\frac{1}{2{\Delta t}}\enorm{\chi^0}^2+\frac{\rho}{2\Delta t}\sum_{n=0}^{m-1}\Lnorm{\dot\theta^{n+1}-\dot\theta^n}{\varpi^{n+1}+\varpi^n}\\\nonumber&-\frac{\rho}{\Delta t}\sum_{n=0}^{m-1}\Lnorm{\dot\theta^{n+1}-\dot\theta^n}{\mathcal{E}_2^n}-\frac{\rho}{\Delta t}\sum_{n=0}^{m-1}\Lnorm{\dot\theta^{n+1}-\dot\theta^n}{\mathcal{E}_3^n}\\\nonumber&+\frac{\rho}{2}\sum_{n=0}^{m-1}\Lnorm{\mathcal{E}_1^n}{\varpi^{n+1}+\varpi^n}-{\rho}\sum_{n=0}^{m-1}\Lnorm{\mathcal{E}_1^n}{\mathcal{E}_2^n}-{\rho}\sum_{n=0}^{m-1}\Lnorm{\mathcal{E}_1^n}{\mathcal{E}_3^n}\\&+\frac{\rho}{\Delta t}\sum_{n=0}^{m-1}\Lnorm{\varpi^{n+1}-\varpi^n}{\mathcal{E}_2^n}+\frac{\rho}{\Delta t}\sum_{n=0}^{m-1}\Lnorm{\varpi^{n+1}-\varpi^n}{\mathcal{E}_3^n}.\label{p1fe2}
\end{align}
In a similar way, we consider the difference of \eqref{weak:p1e2} and \eqref{p1f2}
and obtain
\begin{align}
    \tau_q&\Enorm{\frac{\dot\psi_q^{n+1}+\dot\psi_q^{n}}{2}-\frac{{\Psi}_{hq}^{n+1}-{\Psi}_{hq}^{n}}{\Delta t}}{v}+
    \frac{1}{2}\Enorm{({\psi_q^{n+1}+\psi_q^{n}})-({{\Psi}_{hq}^{n+1}+{\Psi}_{hq}^{n}})}{v}\nonumber\\
    =&\frac{\varphi_q}{2}\Enorm{({u^{n+1}+u^{n}})-({Z_h^{n+1}+Z_h^{n}})}{v}\label{p1fe0_Int}
\end{align}
for each $q$.
Taking $v=(\varsigma_q^{n+1}-\varsigma_q^n)/\Delta t$ in \eqref{p1fe0_Int},
using Galerkin orthogonality with the definitions of
$\varsigma_q^n$, $\chi^n$, $\theta^n$ and $\vartheta_q$,
and then summing over $n=0,\ldots,m-1$ we obtain,
\begin{align}
\frac{\tau_q}{\Delta t^2}
&\sum_{n=0}^{m-1}\enorm{\varsigma_q^{n+1}-\varsigma_q^n}^2+\frac{1}{2\Delta t}\left(\enorm{\varsigma_q^{m}}^2-\enorm{\varsigma_q^{0}}^2\right)-\frac{\varphi_q}{2\Delta t}\sum_{n=0}^{m-1}\Enorm{\chi^{n+1}+\chi^n}{\varsigma_q^{n+1}-\varsigma_q^n}
\nonumber\\
=&\frac{\tau_q}{\Delta t}\sum_{n=0}^{m-1}\Enorm{E_q^n}{\varsigma_q^{n+1}-\varsigma_q^n}
\label{eq:P1_error_proof_1}
\end{align}
where, for each $q$,
\[
E_q(t):=\frac{\dot\psi_q(t+\Delta t)+\dot\psi_q(t)}{2}-\frac{\psi_q(t+\Delta t)-\psi_q(t)}{\Delta t}.
\]
Next, we recall that $\varsigma^0_q=0$ due to the initial condition $\psi_q(0)=0$, and
obtain,
\[\sum_{n=0}^{m-1}\Enorm{\chi^{n+1}+\chi^n}{\varsigma_q^{n+1}-\varsigma_q^n}=2\Enorm{\chi^{m}}{\varsigma_q^{m}}-\sum_{n=0}^{m-1}\Enorm{\chi^{n+1}-\chi^n}{\varsigma_q^{n+1}+\varsigma_q^n},\]and
\[\sum_{n=0}^{m-1}\Enorm{E_q^n}{\varsigma_q^{n+1}-\varsigma_q^n}=\Enorm{E_q^{m-1}}{\varsigma_q^m}-\sum_{n=0}^{m-2}\Enorm{E_q^{n+1}-E_q^n}{\varsigma_q^{n+1}},\]
and then using these in \eqref{eq:P1_error_proof_1} results in
\begin{gather}
\frac{\varphi_q}{2\Delta t}\sum_{n=0}^{m-1}\Enorm{\chi^{n+1}-\chi^n}{\varsigma_q^{n+1}+\varsigma_q^n}
=
\frac{\varphi_q}{\Delta t}\Enorm{\chi^m}{\varsigma_q^m}-\frac{\tau_q}{\Delta t^2}\sum_{n=0}^{m-1}\enorm{\varsigma_q^{n+1}-\varsigma_q^n}^2\nonumber
\nonumber\\
-\frac{1}{2\Delta t}\enorm{\varsigma_q^{m}}^2+\frac{\tau_q}{\Delta t}\Enorm{E_q^{m-1}}{\varsigma_q^{m}}
%\nonumber\\&
-\frac{\tau_q}{\Delta t}    \sum_{n=0}^{m-2}\Enorm{E_q^{n+1}-E_q^n}{\varsigma_q^{n+1}}\label{p1fe2_Int}
\end{gather}
since $\varsigma_q^0=0$ for each $q$.
Using \eqref{p1fe2_Int} in \eqref{p1fe2} and  multiplying by $\Delta t$ leads to,

\begin{multline*}
    \frac{\rho}{2}\llnorm{\varpi^{m}}^2
    +\frac{1}{2}\enorm{\chi^{m}}^2
    \nonumber+\frac{1}{2}\sum_{q=1}^{N_\varphi}\frac{1}{\varphi_q}\enorm{\varsigma_q^{m}}^2+{\Delta t}\sum_{n=0}^{m-1}\sum_{q=1}^{N_\varphi}\frac{\tau_q}{\varphi_q}\enorm{\frac{\varsigma_q^{n+1}-\varsigma_q^n}{\Delta t}}^2
\\
=
\frac{\rho}{2}\llnorm{\varpi^0}^2+\frac{1}{2{}}\enorm{\chi^0}^2+\frac{\rho}{2}\sum_{n=0}^{m-1}\Lnorm{\dot\theta^{n+1}-\dot\theta^n}{\varpi^{n+1}+\varpi^n}
-{\rho}\sum_{n=0}^{m-1}\Lnorm{\dot\theta^{n+1}
-\dot\theta^n}{\mathcal{E}_2^n}
\\
-{\rho}\sum_{n=0}^{m-1}\Lnorm{\dot\theta^{n+1}-\dot\theta^n}{\mathcal{E}_3^n}
+\frac{\rho}{2}\Delta t\sum_{n=0}^{m-1}\Lnorm{\mathcal{E}_1^n}{\varpi^{n+1}+\varpi^n}-{\rho}\Delta t\sum_{n=0}^{m-1}\Lnorm{\mathcal{E}_1^n}{\mathcal{E}_2^n}
\\
-{\rho}\Delta t\sum_{n=0}^{m-1}\Lnorm{\mathcal{E}_1^n}{\mathcal{E}_3^n}
+{\rho}\sum_{n=0}^{m-1}\Lnorm{\varpi^{n+1}-\varpi^n}{\mathcal{E}_2^n}+{\rho}\sum_{n=0}^{m-1}\Lnorm{\varpi^{n+1}-\varpi^n}{\mathcal{E}_3^n}+\sum_{q=1}^{N_\varphi}\Enorm{\chi^m}{\varsigma_q^m}\\
+\sum_{q=1}^{N_\varphi}\frac{\tau_q}{\varphi_q}\Enorm{E_q^{m-1}}{\varsigma_q^{m}}-\sum_{n=0}^{m-2}\sum_{q=1}^{N_\varphi}\frac{\tau_q}{\varphi_q}    \Enorm{E_q^{n+1}-E_q^n}{\varsigma_q^{n+1}},
\end{multline*}
% 
% OLD...
% \begin{align}
%     \frac{\rho}{2}&\llnorm{\varpi^{m}}^2
%     +\frac{1}{2}\enorm{\chi^{m}}^2
%     \nonumber+\frac{1}{2}\sum_{q=1}^{N_\varphi}\frac{1}{\varphi_q}\enorm{\varsigma_q^{m}}^2+{\Delta t}\sum_{n=0}^{m-1}\sum_{q=1}^{N_\varphi}\frac{\tau_q}{\varphi_q}\enorm{\frac{\varsigma_q^{n+1}-\varsigma_q^n}{\Delta t}}^2
%     \\\nonumber
% =&\frac{\rho}{2}\llnorm{\varpi^0}^2+\frac{1}{2{}}\enorm{\chi^0}^2+\frac{\rho}{2}\sum_{n=0}^{m-1}\Lnorm{\dot\theta^{n+1}-\dot\theta^n}{\varpi^{n+1}+\varpi^n}\\\nonumber&-{\rho}\sum_{n=0}^{m-1}\Lnorm{\dot\theta^{n+1}-\dot\theta^n}{\mathcal{E}_2^n}-{\rho}\sum_{n=0}^{m-1}\Lnorm{\dot\theta^{n+1}-\dot\theta^n}{\mathcal{E}_3^n}\\\nonumber&+\frac{\rho}{2}\Delta t\sum_{n=0}^{m-1}\Lnorm{\mathcal{E}_1^n}{\varpi^{n+1}+\varpi^n}-{\rho}\Delta t\sum_{n=0}^{m-1}\Lnorm{\mathcal{E}_1^n}{\mathcal{E}_2^n}-{\rho}\Delta t\sum_{n=0}^{m-1}\Lnorm{\mathcal{E}_1^n}{\mathcal{E}_3^n}\\&+{\rho}\sum_{n=0}^{m-1}\Lnorm{\varpi^{n+1}-\varpi^n}{\mathcal{E}_2^n}+{\rho}\sum_{n=0}^{m-1}\Lnorm{\varpi^{n+1}-\varpi^n}{\mathcal{E}_3^n}+\sum_{q=1}^{N_\varphi}\Enorm{\chi^m}{\varsigma_q^m}\nonumber\\\label{p1fe3}
% &+\sum_{q=1}^{N_\varphi}\frac{\tau_q}{\varphi_q}\Enorm{E_q^{m-1}}{\varsigma_q^{m}}-\sum_{n=0}^{m-2}\sum_{q=1}^{N_\varphi}\frac{\tau_q}{\varphi_q}    \Enorm{E_q^{n+1}-E_q^n}{\varsigma_q^{n+1}},
% \end{align}
%
and in this we note that
$\llnorm{\varpi^0}\leq\llnorm{\dot\theta^0}$ and $\enorm{\chi^0}^2=0$
by the elliptic projection and the choice of discrete initial conditions.

In this Crank-Nicolson method, the terms
$\mathcal{E}_1^n,\mathcal{E}_2^n,\mathcal{E}_3^n, E^n_q$ and $\dot E^n_q$ are
of order $O(\Delta t^2)$.
To see this notice that
\[\frac{\dot w(t_{n+1})+\dot w(t_{n})}{2}-\frac{w(t_{n+1})-w(t_{n})}{\Delta t}=\frac{1}{2\Delta t}\int^{t_{n+1}}_{t_n}w^{(3)}(t)(t_{n+1}-t)(t-t_n)dt,\]
where $w^{(3)}$ is the third time derivative of $w$.
If $w^{(3)}\in L_2(t_n,t_{n+1};L_2(\Omega))$, the Cauchy-Schwarz inequality implies that
\[
\llnorm{\frac{\dot w(t_{n+1})+\dot w(t_{n})}{2}-\frac{w(t_{n+1})-w(t_{n})}{\Delta t}}^2\leq\frac{\Delta t^3}{4}\lVert w^{(3)}\rVert_{L_2(t_n,t_{n+1};L_2(\Omega))}.
\]
Similarly, if $w^{(3)}\in L_2(t_n,t_{n+1};V)$, we can also obtain
\[
\enorm{\frac{\dot w(t_{n+1})+\dot w(t_{n})}{2}-\frac{w(t_{n+1})-w(t_{n})}{\Delta t}}^2\leq C{\Delta t^3}\lVert w^{(3)}\rVert_{L_2(t_n,t_{n+1};V)},
\]
for some positive constant $C$. Therefore, as we are supposing
$u\in H^4(0,T;H^s(\Omega))$, we can derive bounds of order $O(\Delta t^2)$ using 
standard techniques. For example,
\begin{align*}
\left\vert
\Delta t\sum_{n=0}^{m-1}\Lnorm{\mathcal{E}_1^n}{\mathcal{E}_2^n}
\right\vert
\leq &
\frac{\Delta t}{2}\sum_{n=0}^{m-1}\llnorm{\mathcal{E}_1^n}^2+\frac{\Delta t}{2}\sum_{n=0}^{m-1}\llnorm{\mathcal{E}_2^n}^2
% \\
% \leq&
% \frac{\Delta t}{2}\sum_{n=0}^{N-1}\llnorm{\mathcal{E}_1^n}^2+\frac{\Delta t}{2}\sum_{n=0}^{N-1}\llnorm{\mathcal{E}_2^n}^2
\\
\leq&
\frac{\Delta t^4}{8}\Llnorm{u^{(4)}}^2+\frac{\Delta t^4}{8}\Llnorm{\theta^{(3)}}^2.	
\end{align*}
On the other hand, Lemma \ref{lemma:estimate:elliptic} gives us $L_2$ spatial error
estimates for $\theta(t)$ and its time derivatives and so, after we note that the regularity
of internal variables follows that of the solution, we can derive spatial error bounds on
$\mathcal{E}_1^n,\mathcal{E}_2^n,\mathcal{E}_3^n, E^n_q$ and $\dot E^n_q$.

Lastly, using the same techniques as in the proof of Theorem~\ref{thm:fully:stability:p1}, 
we can derive
\begin{align*}
        \max_{0\leq k\leq N}&\llnorm{\varpi^{k}}
    +\max_{0\leq k\leq N}\enorm{\chi^{k}}+
    \sum_{q=1}^{N_\varphi}\max_{0\leq k\leq N}\enorm{\varsigma_q^{k}}+\sqrt{{\Delta t}\sum_{n=0}^{N-1}\sum_{q=1}^{N_\varphi}\frac{\tau_q}{\varphi_q}\enorm{\frac{\varsigma_q^{n+1}-\varsigma_q^n}{\Delta t}}^2}
    \\\nonumber
&\leq
CT\lVert u\rVert_{H^4(0,T;H^{s_2}(\Omega))}(h^{r-1}+\Delta t^2).
\end{align*}
Furthermore, if elliptic regularity holds, as in \eqref{eq:ellipticRegularity:H2},
\begin{align*}
        \max_{0\leq k\leq N}&\llnorm{\varpi^{k}}
    +\max_{0\leq k\leq N}\enorm{\chi^{k}}+
    \sum_{q=1}^{N_\varphi}\max_{0\leq k\leq N}\enorm{\varsigma_q^{k}}+\sqrt{{\Delta t}\sum_{n=0}^{N-1}\sum_{q=1}^{N_\varphi}\frac{\tau_q}{\varphi_q}\enorm{\frac{\varsigma_q^{n+1}-\varsigma_q^n}{\Delta t}}^2}
    \\
&\leq
CT\lVert u\rVert_{H^4(0,T;H^{s_2}(\Omega))}(h^{r}+\Delta t^2).
\end{align*}
In these, $C$ is a positive constant that is independent of $T$, $h$, $\Delta t$, and the exact and numerical solutions.
\end{proof}

The steps that were omitted in the proof above are those that were used in 
Theorem~\ref{thm:fully:stability:p1} to circumvent the need for Gr\"onwall's inequality. 
They were omitted to save space, but can be reconstructed by following the same steps
as in the stability proof.
The main point is that with this extra effort we can state the following error estimate
where the $T$-dependence of the constant is explicit and non-exponential.

\begin{theorem}\label{thm:error:p1}
Suppose that $u\in H^4(0,T;H^{s_2}(\Omega))\cap C^1(0,T;H^{s_2}(\Omega))$, then we have
\begin{align*}
\max_{0\leq k\leq N}\llnorm{\dot u(t_k)-W^k_h}+\max_{0\leq k\leq N}\enorm{u(t_k)-Z^k_h}\leq CT\lVert u\rVert_{H^4(0,T;H^{s_2}(\Omega))}(h^{r-1}+\Delta t^2)
\end{align*}
where $C$ is a positive constant independent of $T$, $h$, $\Delta t$, and the 
exact and numerical
solutions, but dependent on $\Omega$, $\partial \Omega$, $\rho$, $D$ and the
internal variables. If elliptic regularity, \eqref{eq:ellipticRegularity:H2}, holds,
we also have that\begin{align*}
\max_{0\leq k\leq N}\llnorm{\dot u(t_k)-W^k_h}\leq CT\lVert u\rVert_{H^4(0,T;H^{s_2}(\Omega))}(h^{r}+\Delta t^2)
\end{align*}
for a constant $C$ with the same properties as the one above.
\end{theorem}

\begin{proof}
From Lemma \ref{lemma:estimate:p1} we have 
\[
\max_{0\leq k\leq N}\llnorm{\varpi^{k}}+ \max_{0\leq k\leq N}\enorm{\chi^{k}}\leq CT\lVert u\rVert_{H^4(0,T;H^{s_2}(\Omega))}(h^{r-1}+\Delta t^2)
\]
for some positive $C$ with the stated properties.
Combining this with Lemma~\ref{lemma:estimate:elliptic}, we have for any $n$
such that $0\leq n\leq N$,
\begin{align*}
    \enorm{e^n_h}=\enorm{\theta^n-\chi^n}    \leq\enorm{\theta^n}+\enorm{\chi^n}
    \leq CT\lVert u\rVert_{H^4(0,T;H^{s_2}(\Omega))}(h^{r-1}+\Delta t^2),
\end{align*}
and in a similar way, we can also derive
\begin{align*}
\llnorm{\tilde e^n_h}
\leq&CT\lVert u\rVert_{H^4(0,T;H^{s_2}(\Omega))}(h^{r-1}+\Delta t^2).
\end{align*}
Since $n\le N$, it is also true that
\begin{align*}
    \max_{0\leq k\leq N}\llnorm{\tilde e^{k}}+\max_{0\leq k\leq N}\enorm{e^k_h}
    \leq&CT\lVert u\rVert_{H^4(0,T;H^{s_2}(\Omega))}(h^{r-1}+\Delta t^2).
\end{align*}
which proves the first part of the theorem.
If we have elliptic regularity then we conclude,
\begin{align*}
    \max_{0\leq k\leq N}\llnorm{\tilde e^k_h}\leq\max_{0\leq k\leq N}\llnorm{\varpi^{k}}+\max_{0\leq k\leq N}\llnorm{\dot\theta^{k}}
    \leq CT\lVert u\rVert_{H^4(0,T;H^{s_2}(\Omega))}(h^{r}+\Delta t^2)
\end{align*}
and this now completes the proof.
\end{proof}

With elliptic regularity we can obtain an improved estimate for 
$u(t_k)-Z^k_h$, as shown in the following corollary.

\begin{corollary}\label{cor:error:p1}
Under same conditions as \textnormal{Theorem \ref{thm:error:p1}}, if elliptic regularity holds, then\begin{align*}
\max_{0\leq k\leq N}\llnorm{u(t_k)-Z^k_h}\leq CT\lVert u\rVert_{H^4(0,T;H^{s_2}(\Omega))}(h^{r}+\Delta t^2).
\end{align*}
\end{corollary}

\begin{proof}
From the error splitting and the coercivity of the bilinear form, 
$a(\cdot,\cdot)$, we have,
\[
\llnorm{u(t_n)-Z^n_h}=\llnorm{e^n_h}\leq\llnorm{\theta^n}+\llnorm{\chi^n}
\leq
\llnorm{\theta^n}+\frac{1}{\sqrt{\kappa}}\enorm{\chi^n}
\]
for any $0\leq n\leq N$. Also, from
Lemmas~\ref{lemma:estimate:elliptic} and~\ref{lemma:estimate:p1} we get 
$\llnorm{\theta^n}\leq C\lvert u^n\rvert_{H^r(\Omega)}h^{r}$ 
and $\enorm{\chi^n}\leq CT\lVert u\rVert_{H^4(0,T;H^{s_2}(\Omega))}(h^{r}+\Delta t^2)$.
We therefore obtain
\begin{align*}
\max_{0\leq k\leq N}\llnorm{u(t_k)-Z^k_h}\leq CT\lVert u\rVert_{H^4(0,T;H^{s_2}(\Omega))}(h^{r}+\Delta t^2)
\end{align*}
as claimed.
\end{proof}

This completes the analysis of the displacement form of the problem. We now
move on to the velocity form.

%\subsection{Velocity form}\label{subsec:discveloform}\noindent
\section{Velocity form}\label{sec:discveloform}\noindent
Recall the variational formulation of the velocity form \eqref{weak:p2e1} --- \eqref{weak:p2e2}. Again we adopt a Crank-Nicolson type of time discretization and pose the fully discrete
formulation for \textbf{(P2)} as follows.

$\mathrm{(\mathbf{P2})}^h$
Find $Z^n_h$, $W^n_h$, $\mathcal{S}_{h1}^n,\ \mathcal{S}_{h2}^n,\ldots,\mathcal{S}_{hN_\varphi}^n\in V^h$ for $n=0,\ldots,N$ such that \eqref{r1} holds and
\begin{align}
\Lnorm{\rho\frac{W^{n+1}_h-W^{n}_h}{\Delta t}}{v}
+\varphi_0a\left(\frac{Z_h^{n+1}+Z_h^{n}}{2},v\right)
+ &\sum\limits_{q=1}^{N_\varphi}
a\left(\frac{\mathcal{S}_{hq}^{n+1}+\mathcal{S}_{hq}^{n}}{2},v\right)
\nonumber\\
&=\frac{1}{2}( F_v(t_{n+1};v)+ F_v(t_{n};v)),\label{p2f1}
\\
\tau_{q}a\left(\frac{\mathcal{S}_{hq}^{n+1}-\mathcal{S}_{hq}^{n}}{\Delta t},v\right)
+a\left(\frac{\mathcal{S}_{hq}^{n+1}+\mathcal{S}_{hq}^{n}}{2},v\right)
&=\tau_{q}\varphi_qa\left(\frac{W_h^{n+1}+W_h^{n}}{2},v\right)\ \textrm{for each }{q},
\label{p2f2}
\\
a(Z_h^0,{v})
&=a({u_0},{v}),
\label{p2f3}\\
\Lnorm{W^0_h}{v}
&=\Lnorm{w_0}{v},
\label{p2f4}\\
\mathcal{S}_{hq}^0
&=0,\  \text{for each $q$,}
\label{p2f5}
\end{align}
each for all $v\in V^h$.

We first give a stability estimate. The proof is similar to that for the displacement form, Theorem~\ref{thm:fully:stability:p1}, and so most of the steps are
not included here.

\begin{theorem}\label{thm:fully:stability:p2}
If $f\in C(0,T;L_2(\Omega))$, $g_N\in H^1(0,T;L_2(\Gamma_N))\cap C(0,T;L_2(\Gamma_N))$ and $u_0\in H^1(\Omega)$,
then $\mathrm{(\mathbf{P2})}^h$ has a unique solution and there exists a positive constant $C$ depending on $\Omega,\ \partial\Omega$ and the sets $\{\varphi_q\}_{q=0}^{N_\varphi}$ and $\{\tau_q\}_{q=1}^{N_\varphi}$, but independent of the 
exact and numerical solutions, $h$, $\Delta t$ and $T$ such that,

\begin{align*}
&\max_{0\leq n\leq N}\llnorm{W^{n}_h}^2+\max_{0\leq n\leq N}\enorm{Z^{n}_h}^2+\sum_{q=1}^{N_\varphi}\max_{0\leq n\leq N}\enorm{\mathcal{S}_{hq}^{n}}^2+\sum_{q=1}^{N_\varphi}\sum_{n=0}^{N-1}{\Delta t}\enorm{\mathcal{S}_{hq}^{n+1}+\mathcal{S}_{hq}^{n}}^2
\\
&\leq
~CT^2\bigg(\llnorm{w_0}^2+\enorm{u_0}^2+\ilnorm{{f}}^2+\hgnorm{g_N}^2\bigg).
\end{align*}
% 
% \begin{align*}
% &\max_{0\leq n\leq N}\llnorm{W^{n}_h}^2+\max_{0\leq n\leq N}\enorm{Z^{n}_h}^2+\sum_{q=1}^{N_\varphi}\max_{0\leq n\leq N}\enorm{\mathcal{S}_{hq}^{n}}^2+\sum_{q=1}^{N_\varphi}\sum_{n=0}^{N-1}{\Delta t}\enorm{\mathcal{S}_{hq}^{n+1}+\mathcal{S}_{hq}^{n}}^2
% \\
% &\leq
% ~CT^2\bigg(\llnorm{w_0}^2+\enorm{u_0}^2+\ilnorm{{f}}^2+\ignorm{g_N}^2
% \\
% &\qquad\qquad\qquad+\lgnorm{\dot g_N}^2\bigg).
% \end{align*}
% 
\end{theorem}

\begin{proof}
As before, once we prove the stated stability bound,
we can conclude that the linear system \eqref{p2f1}--\eqref{p2f5} has a unique
solution. So, we take $v=W^{n+1}_h+W^n_h$ in \eqref{p2f1} and
$v=\mathcal{S}_{hq}^{n+1}+\mathcal{S}_{hq}^{n}$ in \eqref{p2f2} for each $q$,
and sum over time to see that for $1\leq m\leq N$,
\begin{align}
{\rho}&\llnorm{W^{m}_h}^2+{\varphi_0}\enorm{Z^{m}_h}^2+\sum_{q=1}^{N_\varphi}\frac{1}{\varphi_q}\enorm{\mathcal{S}_{hq}^{m}}^2+\sum_{q=1}^{N_\varphi}\sum_{n=0}^{m-1}\frac{\Delta t}{2\tau_{q}\varphi_q}\enorm{\mathcal{S}_{hq}^{n+1}+\mathcal{S}_{hq}^{n}}^2\nonumber\\=&\rho\llnorm{W^0_h}^2+\varphi_0\enorm{Z^0_h}^2+\sum_{n=0}^{m-1}\frac{\Delta t}{2} \left(F_v\left(t_{n+1};W^{n+1}_h+W^n_h\right)+F_v\left(t_{n};W^{n+1}_h+W^n_h\right)\right).
\end{align}
We now follow the proof of Theorem \ref{thm:fully:stability:p1}, particularly
\eqref{pf:sta:p1:eq2}, the only main difference here being that $F_v$ also
includes
$-\sum\limits_{{q}=1}^{N_\varphi}\varphi_qe^{-t/\tau_{q}}a(u_0,v)$
(see under \eqref{weak:p2e2}).
This term is easily dealt with after we recall that $\varphi(0)=1$
and note that $0<e^{-t/\tau_{q}}\leq 1$ for $t\geq 0$ and for each $q$.
In fact,
\[
\left\vert
-\sum\limits_{{q}=1}^{N_\varphi}\varphi_qe^{-t/\tau_{q}}a(u_0,v)
\right\vert
\leq 
\sum\limits_{{q}=1}^{N_\varphi}\varphi_qe^{-t/\tau_{q}}\enorm{u_0}\enorm{v}\\
\leq
\sum\limits_{{q}=1}^{N_\varphi}\varphi_q\enorm{u_0}\enorm{v}
\leq\enorm{u_0}\enorm{v}.
\]
The proof can now be completed using similar arguments as for the proof of Theorem \ref{thm:fully:stability:p1}.
\end{proof}

Error estimates for $\mathrm{(\mathbf{P2})}^h$ can now be derived using similar
steps to those in Lemma \ref{lemma:estimate:p1} and Theorem \ref{thm:error:p1}.
Before the first main result we introduce this addtional notation for the
error splitting:
$\zeta_q - \mathcal{S}_{hq}^n = (\zeta_q-R\zeta_q) - \Upsilon^n_q$, for 
$\Upsilon^n_q:=\mathcal{S}_{hq}^n-R\zeta_q^n$ for $q=1,\ldots,N_\varphi$.
Only the $\Upsilon^n_q$ terms need estimating here, as we can use Galerkin orthogonality
on the other error terms.

\begin{lemma}\label{lemma:estimate:p2}
Suppose that $u\in H^4(0,T;H^{s_2}(\Omega))\cap C^1(0,T;H^{s_2}(\Omega))$,
then there exists a positive constant $C$ such that
\begin{gather*}
\max_{0\leq k\leq N}\llnorm{\varpi^{k}}+\max_{0\leq k\leq N}\enorm{\chi^{k}}+\sum_{q=1}^{N_\varphi}\max_{0\leq k\leq N}\enorm{\Upsilon_q^{k}}+\sqrt{\Delta t\sum_{n=0}^{N-1}\sum_{q=1}^{N_\varphi}\enorm{\Upsilon_q^{n+1}+\Upsilon_q^n}^2}\\
\leq CT\lVert u\rVert_{H^4(0,T;H^{s_2}(\Omega))}(h^{r-1}+\Delta t^2).
\end{gather*}
Furthermore, if we assume elliptic regularity as in \eqref{eq:ellipticRegularity:H2},
we also have

\begin{gather*}
\max_{0\leq k\leq N}\llnorm{\varpi^{k}}+\max_{0\leq k\leq N}\enorm{\chi^{k}}+\sum_{q=1}^{N_\varphi}\max_{0\leq k\leq N}\enorm{\Upsilon_q^{k}}+\sqrt{\Delta t\sum_{n=0}^{N-1}\sum_{q=1}^{N_\varphi}\enorm{\Upsilon_q^{n+1}+\Upsilon_q^n}^2}\\
\leq CT\lVert u\rVert_{H^4(0,T;H^{s_2}(\Omega))}(h^{r}+\Delta t^2).
\end{gather*}
In these the constant $C$ is independent of $h$, $\Delta t$, $T$ and the exact and 
numerical solutions, but depends on $\rho$, $D$, $\Omega$, $\partial \Omega$ and
the internal variable coefficients.
\end{lemma}

\begin{proof}
The proof follows similar steps to that of Lemma \ref{lemma:estimate:p1} and so we don't need to give all of the details. First, taking averages at $t_{n+1}$ and $t_n$ of
\eqref{weak:p2e1} and \eqref{weak:p2e2}, and subtracting from \eqref{p2f1} and
\eqref{p2f2} with the test functions
$v = (\chi^{n+1} - \chi^n) / \Delta t\in V^h$ and 
$v=(\Upsilon_q^{n+1}+\Upsilon_q^n)/2\in V^h$,
we have for any $0\leq n\leq N-1$ that,
\begin{align}
\frac{\rho}{2}&\left(\llnorm{\varpi^{n+1}}^2-\llnorm{\varpi^n}^2\right)+\frac{\varphi_0}{2}\left(\enorm{\chi^{n+1}}^2-\enorm{\chi^n}^2\right)\nonumber\\\nonumber&+\frac{1}{2{}}\sum_{q=1}^{N_\varphi}\frac{1}{\varphi_q }\left(\enorm{\Upsilon_q^{n+1}}^2-\enorm{\Upsilon_q^n}^2\right)+\frac{\Delta t}{2{}}\sum_{q=1}^{N_\varphi}\enorm{\Upsilon_q^{n+1}+\Upsilon_q^n}^2\\
=&\frac{\rho}{2}\Lnorm{\dot\theta^{n+1}-\dot\theta^n}{\varpi^{n+1}+\varpi^n}+\frac{\rho}{2}{\Delta t}\Lnorm{\mathcal{E}_1^n}{\varpi^{n+1}+\varpi^n}\nonumber\\&-{\rho}\Lnorm{\dot\theta^{n+1}-\dot\theta^n}{\mathcal{E}_2^n}-{\rho}\Lnorm{\dot\theta^{n+1}-\dot\theta^n}{\mathcal{E}_3^n}-{\rho}{\Delta t}\Lnorm{\mathcal{E}_1^n}{\mathcal{E}_2^n}\nonumber\\&-{\rho}{\Delta t}\Lnorm{\mathcal{E}_1^n}{\mathcal{E}_3^n}+\frac{\Delta t}{2}\sum_{q=1}^{N_\varphi}\frac{1}{\varphi_q}\Enorm{E^n_q}{\Upsilon_q^{n+1}+\Upsilon_q^n}-\frac{\Delta t}{2}\sum_{q=1}^{N_\varphi}\frac{1}{\varphi_q}\Enorm{\mathcal{E}_3^n}{\Upsilon_q^{n+1}+\Upsilon_q^n}\label{p2fe3}
\end{align}
where
\begin{align*}
\mathcal{E}_1(t):=\frac{\ddot u(t+\Delta t)+\ddot u(t)}{2}-\frac{\dot u(t+\Delta t)-\dot u(t)}{\Delta t},\ 
    \mathcal{E}_2(t):=\frac{\dot \theta(t+\Delta t)+\dot \theta(t)}{2}-\frac{\theta(t+\Delta t)-\theta(t)}{\Delta t},\\
    \mathcal{E}_3(t):=\frac{u(t+\Delta t)-u(t)}{\Delta t}-\frac{\dot u(t+\Delta t)+\dot u(t)}{2},\
    E_q(t):=\frac{\dot\zeta(t+\Delta t)+\dot\zeta(t)}{2}-\frac{\zeta(t+\Delta t)-\zeta(t)}{\Delta t},
\end{align*}for each $q$. The remainder of the proof can be completed in 
much the same way as that of Lemma \ref{lemma:estimate:p1}, with a careful choice
of the `$\epsilon$' parameter in the Young's inequalities.
\end{proof}

\begin{theorem}\label{thm:error:p2}
Suppose that $u\in H^4(0,T;H^{s_2}(\Omega))\cap C^1(0,T;H^{s_2}(\Omega))$, 
then we have
\begin{align*}
\max_{0\leq k\leq N}\enorm{u(t_k)-Z^k_h}+
\max_{0\leq k\leq N}\llnorm{\dot u(t_k)-W^k_h}\leq CT\lVert u\rVert_{H^4(0,T;H^{s_2}(\Omega))}(h^{r-1}+\Delta t^2)
\end{align*}
for some positive $C$ that is independent of $h$, $\Delta t$, and the exact and
numerical solutions but dependent on $\rho$, $D$, $\Omega$, $\partial \Omega$, $T$
and the internal variable coefficients. In addition,
\begin{align*}
\max_{0\leq k\leq N}\llnorm{\dot u(t_k)-W^k_h}\leq CT\lVert u\rVert_{H^4(0,T;H^{s_2}(\Omega))}(h^{r}+\Delta t^2)
\end{align*}
if elliptic regularity, \eqref{eq:ellipticRegularity:H2}, can be assumed.
\end{theorem}
\begin{proof}
Proceeding similarly to the proof of Theorem~\ref{thm:error:p1},
by Lemma \ref{lemma:estimate:p2} we have
\[
\max_{0\leq k\leq N}\llnorm{\varpi^{k}}+ \max_{0\leq k\leq N}\enorm{\chi^{k}}\leq CT\lVert u\rVert_{H^4(0,T;H^{s_2}(\Omega))}(h^{r-1}+\Delta t^2)
\]
and then Lemmas~\ref{lemma:estimate:elliptic} and~\ref{lemma:estimate:p2}, and the triangle inequality give,
% 
% and then similarly to the proof of Theorem \ref{thm:error:p1}, the triangle inequality
% and Lemmas \ref{lemma:estimate:elliptic} and \ref{lemma:estimate:p2} give,
\begin{align*}
    \max_{0\leq k\leq N}\enorm{e^k_h}
    +\max_{0\leq k\leq N}\llnorm{\tilde e^k_h}
    \leq CT\lVert u\rVert_{H^4(0,T;H^{s_2}(\Omega))}(h^{r-1}+\Delta t^2).
\end{align*}
Further, if elliptic regularity holds
\begin{align*}
    \max_{0\leq k\leq N}\llnorm{\tilde e^k_h}
    \leq&CT\lVert u\rVert_{H^4(0,T;H^{s_2}(\Omega))}(h^{r}+\Delta t^2).
\end{align*}
This completes the proof.
\end{proof}

In analogy to Corollary~\ref{cor:error:p1} we can also show an improved estimate.

\begin{corollary}\label{cor:error:p2}
Under the same conditions as \textnormal{Theorem \ref{thm:error:p2}}, if 
elliptic regularity, \eqref{eq:ellipticRegularity:H2}, holds then,
\begin{align*}
\max_{0\leq k\leq N}\llnorm{u(t_k)-Z^k_h}\leq CT\lVert u\rVert_{H^4(0,T;H^{s_2}(\Omega))}(h^{r}+\Delta t^2),
\end{align*}
for a positive constant $C$ with the same qualities as above.
\end{corollary}
\begin{proof}
The proof is parallel to that of Corollary \ref{cor:error:p1} but instead of using the
result from Theorem \ref{thm:error:p1}, we use Theorem \ref{thm:error:p2} instead.
\end{proof}

\section{Numerical experiments}\label{sec:numexp}\noindent
In this section we give some evidence that the convergence rates given in the theorems
above are realised in practice, at least for model problems for which an exact solution
can be generated. The tabulated results in this section can be reproduced by the
python scripts (using FEniCS, \cite{alnaes2015fenics}, 
\texttt{https://fenicsproject.org/}) at
\texttt{https://github.com/Yongseok7717} or by pulling and running a custom 
docker container as follows (at a \texttt{bash} prompt):
\begin{verbatim}
docker pull variationalform/fem:yjcg1
docker run -ti variationalform/fem:yjcg1
cd ./2019-11-26-codes/mainTable/
./main_Table.sh
\end{verbatim}
The run may take around 30 minutes or longer depending on the host machine.

Let the exact solution to \eqref{eq:momeq} -- \eqref{eq:initial:velo} be
\[
u(x,y,t)=e^{-t}\sin(xy)\in C^\infty(0,T;C^\infty(\Omega))
\]
where $\Omega$ is the unit square, $(0,1)\times(0,1)$, and $T=1$.
The Dirichlet boundary condition is given by
$u=0$ if $x=0$ or $y=0$ (for all $t$),
with Neumann data on the remainder of $\partial\Omega$.
Furthermore, we take two internal variables with $\varphi_0=0.5,\ \varphi_1=0.1,\ \varphi_2=0.4,\    \tau_1=0.5,\ \tau_2=1.5$, and set $\rho=D=1$.
The internal variables, the source term $f$ and the Neumann boundary term $g_N$
are all determined by inserting the exact solution above into the governing equations.

From the error estimates, regardless of whether we use the displacement or velocity form
of the internal variables, for this smooth solution we can expect that
$\enorm{ e_h^N}$, $\llnorm{\tilde e_h^N}$, $\llnorm{ e_h^N}$, respectively, are of 
optimal orders $O(h^{s_1}+\Delta t^2)$,
$O(h^{s_1+1}+\Delta t^2)$, $O(h^{s_1+1}+\Delta t^2)$ respectively (because $s_2=\infty$).
In other words, the convergence rate with respect to time is fixed at second order but the spatial convergence order depends on the degree of polynomials $s_1$ used in the finite element space $V^h$.

If we take $\Delta t\sim h$ then we will expect for our errors that
\[
\enorm{ e_h^N}=O(h^{\min(s_1,2)}), \qquad\llnorm{\tilde e_h^N},\llnorm{ e_h^N}=O(h^{\min(s_1+1,2)}).
\]
In the computational results that follow, the numerical convergence rate, $d_c$,
is estimated by
\[
d_c\ =\ \frac{\log(\text{error of }h_1)-\log(\text{error of }h_2)}{\log(h_1)-\log(h_2)}.
\]
We can see this in Figure~\ref{fig:convergentOrder} where on the left we give
results for a piecewise linear basis ($s_1=1$) and on the right hand for a quadratic basis
($s_1=2$). The convergence rate is given by the gradients of the lines.
With linears the energy errors have first order accuracy but the $L_2$ errors show optimal second order. On the other hand, for quadratics, we can observe second order rates 
for all quantities because the time-error convergence order is fixed at $2$.
\begin{figure}[H]
	\centering
	\centering
	\includegraphics[width=0.45\textwidth]{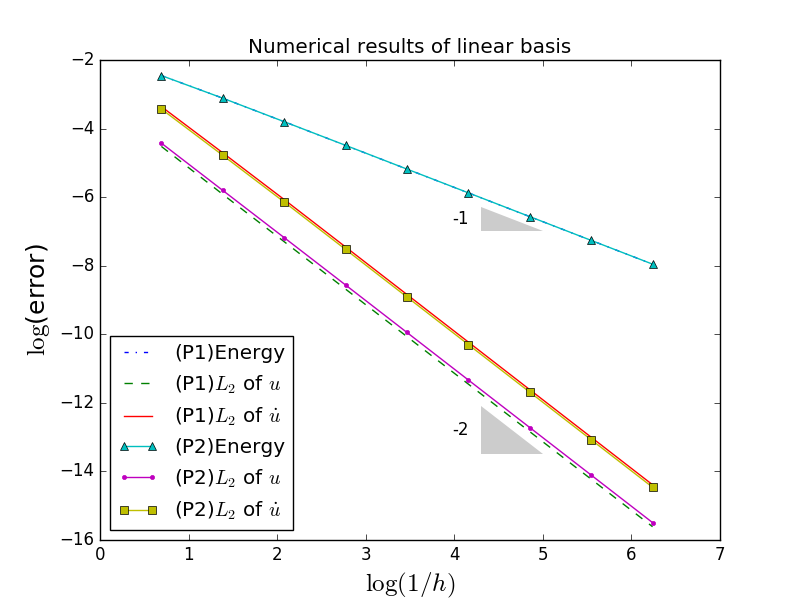}
	\includegraphics[width=0.45\textwidth]{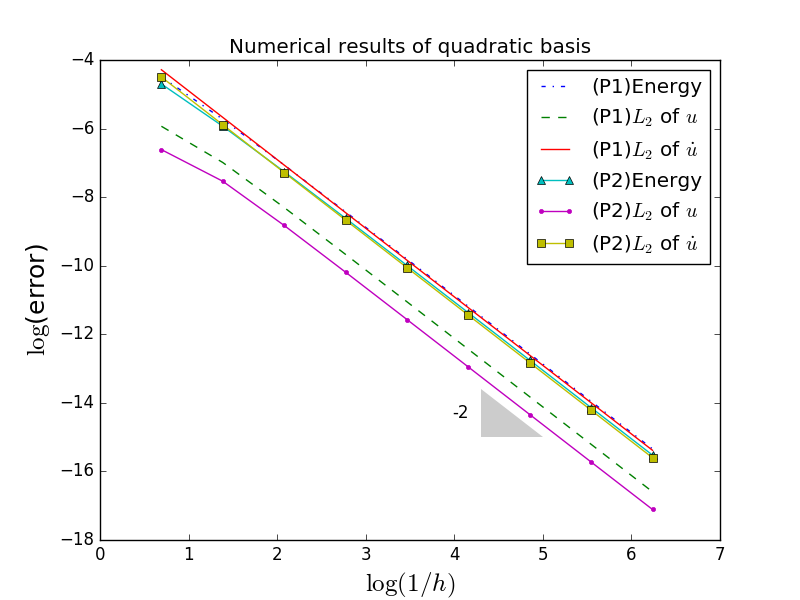}
	\caption{Numerical convergent order: \textbf{linear} (left) and \textbf{quadratic} (right) polynomial basis}
	\label{fig:convergentOrder}
\end{figure}
To see the higher order spatial convergence for the quadratic case we take a much
smaller $\Delta t$ so that we can consider the time error to be negligible in that,
\[
\enorm{ e_h^N}\approx O(h^2)\textrm{ and } 
\llnorm{\tilde e_h^N},
\llnorm{ e_h^N}\approx O(h^{3}).
\]
The results are shown in Table \ref{table:fixedTime}.
In the same way, if we make the spatial error negligible, we can observe the temporal error
convergence rate --- and this is given in Table \ref{table:fixed_h}. We see 
that the rates are optimal in all cases.

\begin{table}[H]
	\centering
	\begin{tabular}{|c||ccc|ccc|}
		\hline
		\multirow{2}{*}{$h$}&\multicolumn{3}{c|}{\textbf{Displacement form}}&\multicolumn{3}{c|}{\textbf{Velocity form}}\\		
		&$\enorm{ e_h^N}$&$\llnorm{\tilde e_h^N}$&$\llnorm{ e_h^N}$&$\enorm{ e_h^N}$&$\llnorm{\tilde e_h^N}$&$\llnorm{ e_h^N}$\\
		\hline
		\hline
		1/4&2.2557E{-3}&8.1101E{-5}&6.9417E{-5}&2.2557E{-3}&8.1098E{-5}&6.9419E{-5}\\
		1/8&6.0301E{-4}&1.0491E{-5}&9.2260E{-6}&6.0301E{-4}&1.0489E{-5}&9.2266E{-6}\\
		1/16&1.5566E{-4}&1.2803E{-6}&1.1954E{-6}&1.5566E{-4}&1.2794E{-6}&1.1957E{-6}\\
		1/32&3.9526E{-5}&1.6460E{-7}&1.5240E{-7}&3.9526E{-5}&1.6270E{-7}&1.5226E{-7}\\
		\hline
		rate&1.93&    2.99&    2.93&    1.93&    2.99&    2.93\\
		\hline
	\end{tabular}
	\caption{Fixed time step size errors when $s_1=2$ and $\Delta t=1/1200$}\label{table:fixedTime}
\end{table}

\begin{table}[H]
	\centering
	\begin{tabular}{|c||ccc|ccc|}
		\hline
		\multirow{2}{*}{$\Delta t$}&\multicolumn{3}{c|}{\textbf{Displacement form}}&\multicolumn{3}{c|}{\textbf{Velocity form}}\\		
		&$\enorm{ e_h^N}$&$\llnorm{\tilde e_h^N}$&$\llnorm{ e_h^N}$&$\enorm{ e_h^N}$&$\llnorm{\tilde e_h^N}$&$\llnorm{ e_h^N}$\\
		\hline
		\hline
1/8& 6.0705E-04& 8.5271E-04& 2.4904E-04& 3.6453E-04& 6.8608E-04& 1.4780E-04 \\
1/16& 1.5316E-04& 2.1327E-04& 6.3192E-05& 9.2174E-05& 1.7163E-04& 3.7643E-05 \\
1/32& 3.8373E-05& 5.3325E-05& 1.5856E-05& 2.3105E-05& 4.2915E-05& 9.4542E-06 \\
1/64& 9.5993E-06& 1.3332E-05& 3.9677E-06& 5.7818E-06& 1.0729E-05& 2.3663E-06 \\
		\hline
		rate& 1.99 &   2.00 &   1.99  &  1.99 &   2.00  &  1.98
		\\
		\hline
	\end{tabular}
	\caption{Fixed spatial mesh size errors when $s_1=2$ and $h=1/512$}\label{table:fixed_h}
\end{table}
In summary, both the displacement form and the velocity form display numerical
results consistent with the given error bounds. This is true for the $L_2(\Omega)$
estimates even though the elliptic regularity estimate \eqref{eq:ellipticRegularity:H2}
is usually only relied upon for homogeneous Dirichlet problems.
We included a Neumann boundary condition here for generality but the
code could easily be altered to the pure Dirichlet case to conform to the standard requirements for
elliptic regularity, although we note that in
\cite[Chapter 4.3]{grisvard2011elliptic} and \cite{CostabelDaugeNicaise2012} there are
discussions of elliptic regularity for problems where both Dirichlet and Neumann
boundary conditions are present.
In any event, our intention was simply to demonstrate that the optimal
rates are in fact achieved in practice.

\section{Conclusions}\label{sec:concs}
Our two fully discrete formulations demonstrate optimal energy and $L_2$ spatial error 
estimates, and second order temporal error estimates, both in theory, and
in numerical tests. We took the usual step in assuming ideal conditions for the
proofs although in practical problems one cannot always expect such optimality.
For example, if we had lower spatial regularity due, say, to corner singularities we
would expect the energy estimates to be of the order $O(h^{\min(s_1+1,s_2)-1}+\Delta t^2)$
where the specific values of the exponents would depend on the geometry and strength of
the singularity. Also, with such reduced regularity we would be unlikely to have the
necessary elliptic regularity for a higher order $L_2$ estimate.

If, on the other hand, the regularity in
time was reduced then we would expect to see $\Delta t^2$ replaced by $\Delta t^\beta$,
for $\beta < 2$, in the above. This may stem from the loads $f$ and $g$ being non-smooth
in time, possibly even discontinuous in some applications (we can think of intermittent hammer blows on a structure for example). Our assumptions on the temporal behaviour
of $f$ and $g$ are quite strong, and these allowed us to circumvent the use of
Gr{\"o}nwall's lemma. Although it would be an interesting to see how these
assumptions could be relaxed while retaining the sharper bounds, we have no choice here
but to defer this to a later study.

This scalar, or antiplane shear, problem studied above can can be straighforwardly
elevated to a vector-valued problem representing dynamic linear viscoelasticity by making
some notational changes and using product Hilbert spaces. For details see, for example,
\cite{DGV,Shaw11i} but, in brief, for this we would define the
Cauchy infinitesimal strain tensor
${{\varepsilon}}_{ij}(\boldsymbol{u})
=\frac{1}{2}\left(\frac{\partial u_i}{\partial x_j}+\frac{\partial u_j}{\partial x_i}\right)$ for $i,j=1,\ldots,d$, $d=2$ or $3$ and where $\boldsymbol{u}$ is the displacement vector.
The tensor $\ushort{\boldsymbol\varepsilon}(\boldsymbol{u})$ then plays the 
role of $\nabla u$. We then replace $D$ with $\ushort{\boldsymbol D}$,
a symmetric positive definite fourth order tensor, and the stress tensor
$\ushort{\boldsymbol{\sigma}}$ is given in direct analogy to the scalar analogues in
\eqref{eq:dispconlaw} and \eqref{eq:veloconlaw}.
The resulting variational formulation uses the symmetric bilinear form
$a(\boldsymbol{w},\boldsymbol{v})=\Lnorm{\ushort{\boldsymbol D} \ushort{\boldsymbol{\varepsilon}}(\boldsymbol{w})} {\ushort{\boldsymbol{\varepsilon}}(\boldsymbol{v})}$ for $\boldsymbol{w},\boldsymbol{v}\in \boldsymbol{V}$ where \mbox{$\boldsymbol{V}=\{\boldsymbol{v}\in[H^1(\Omega)]^d\ |\ \boldsymbol v(\boldsymbol{x})=0 \textrm{ on }\Gamma_D\}$}. The coercivity of this form follows from Korn's inequality (e.g. \cite{ciarlet2010korn,horgan1983inequalities,nitsche1981korn,MTE}), and continuity follows from the Cauchy-Schwarz inequality. Internal variables can be defined by exact analogy with those
above, and continous and discrete variational problems can similarly posed. The stability
and error analyses then go through in the same way as above with the help of 
the elasticity theory estimates in \cite{brenner1992linear,MTE,amara1979equilibrium},
and we will obtain similarly optimal bounds without the use of Gr\"onwall's inequality.

\bibliography{ref}

\bibliographystyle{elsarticle-num}

\end{document}